\documentclass[dvipdfmx,12pt]{article}

\usepackage{amssymb,amscd,amsthm,amsmath}
\usepackage{graphicx}
\usepackage{tikz}
\usepackage{float}
\usetikzlibrary{cd}
\usetikzlibrary{intersections,calc,arrows.meta}
\usepackage{mathrsfs}

\newtheorem{theorem}{Theorem}[section]
\newtheorem{lemma}[theorem]{Lemma}
\newtheorem{proposition}[theorem]{Proposition}
\newtheorem{cor}[theorem]{Corollary}

\theoremstyle{definition}
\newtheorem{definition}[theorem]{Definition}
\newtheorem{example}[theorem]{Example}
\newtheorem{remark}[theorem]{Remark}

\numberwithin{equation}{section}

\newcommand{\B}{\mathbb{B}}
\newcommand{\C}{\mathbb{C}}

\newcommand{\K}{\mathbb{K}}
\newcommand{\M}{\mathbb{M}}
\newcommand{\N}{\mathbb{N}}

\newcommand{\R}{\mathbb{R}}
\newcommand{\T}{\mathbb{T}}
\newcommand{\Z}{\mathbb{Z}}
\newcommand{\cA}{\mathcal{A}}
\newcommand{\cB}{\mathcal{B}}
\newcommand{\cC}{\mathcal{C}}
\newcommand{\cD}{\mathcal{D}}
\newcommand{\cG}{\mathcal{G}}

\newcommand{\cH}{\mathcal{H}}
\newcommand{\cK}{\mathcal{K}}
\newcommand{\cL}{\mathcal{L}}
\newcommand{\cO}{\mathcal{O}}

\newcommand{\cE}{\mathcal{E}}
\newcommand{\cF}{\mathcal{F}}

\newcommand{\cU}{\mathcal{U}}

\newcommand\Aut{\operatorname{Aut}}

\newcommand\End{\operatorname{End}}
\newcommand\Mor{\operatorname{Mor}}

\newcommand\Tr{\operatorname{Tr}}
\newcommand\id{\operatorname{id}}
\newcommand\Ad{\operatorname{Ad}}
\newcommand\Hom{\operatorname{Hom}}

\newcommand{\Ind}{\operatorname{Ind}}

\newcommand\Rep{\operatorname{Rep}}
\newcommand\inpr[2]{\langle{#1,#2}\rangle}

\newcommand{\biota}{\overline{\iota}}
\newcommand{\brho}{\overline{\rho}}
\newcommand{\bpi}{\overline{\pi}}
\newcommand{\bR}{\overline{R}}

\newcommand{\hG}{\widehat{G}}
\newcommand{\halpha}{\hat{\alpha}}
\newcommand{\calpha}{\check{\alpha}}

\newcommand{\tH}{\widetilde{\cH}}
\newcommand{\talpha}{\widetilde{\alpha}}
\newcommand{\bu}{\overline{u}}
\newcommand{\bpsi}{\overline{\psi}}
\title{Minimal compact group actions on C$^*$-algebras with simple fixed point algebras}

\author{Masaki Izumi
\thanks{Supported in part by JSPS KAKENHI Grant Number JP20H01805}\\
Graduate School of Science \\
Kyoto University \\
Sakyo-ku, Kyoto 606-8502, Japan} 

\begin{document} 
\maketitle
\centerline{In memory of the late Professor Huzihiro Araki} 
\begin{abstract} The notion of qausi-product actions of a compact group on a C$^*$-algebra was introduced 
by Bratteli et al. in their attempt to seek an equivariant analogue of Glimm's characterization 
of non-type I C$^*$-algebras. 
We show that a faithful minimal action of a second countable compact group on a separable C$^*$-algebra 
is quasi-product whenever its fixed point algebra is simple.  
This was previously known only for compact abelian groups and for profinite groups. 
Our proof relies on a subfactor technique applied to finite index inclusions of simple C$^*$-algebras in the purely infinite case, 
and also uses ergodic actions of compact groups in the general case.  
As an application, we show that if moreover the fixed point algebra is a Kirchberg algebra, such an action is 
always isometrically shift-absorbing, and hence is classifiable by the equivariant KK-theory due to a recent 
result of Gabe-Szab\'o. 
\end{abstract}
\section{Introduction} 
Compact group actions on C$^*$-algebras have been extensively studied, especially in relation to mathematical physics 
in the early history of operator algebras, as internal symmetries in models in physics are often described by 
compact group actions (see, for example, \cite{DR89II}, \cite{DR90}).  
Araki-Haag-Kastler-Takesaki's work \cite{AHKT77} on chemical potential is one of the highlights in this subject in the 70's, 
which inspired a lot of subsequent work both in mathematics and physics.

In the pure mathematics side, the most notable class of compact group actions on C$^*$-algebras is probably quasi-product actions 
introduced by Bratteli et al. \cite{BER87}, \cite{BEEK89}, \cite{BEK93}. 
Roughly speaking, while the classical Glimm theorem characterizes non-type I C$^*$-algebras as those having sub-quotients 
isomorphic to the CAR algebra (see \cite[Theorem 6.7.3]{Pe} for the precise statement), quasi-product actions are, by definition, 
those having invariant-quotients equivariantly isomorphic to the UHF algebras with product type actions. 
\cite[Theorem 1]{BEK93} shows that ten conditions on a compact group action are mutually equivalent, 
which can be taken as the definition (and characterization) of a quasi-product action. 
The existence of an invariant pure state is one of them, which is very useful for practical applications 
because it assures that every invariant state can be approximated by invariant pure states (see \cite{BKR97}). 

A Galois correspondence for minimal actions of compact groups on factors was established by Izumi-Longo-Popa \cite{ILP98}. 
Its C$^*$-counterpart was recently obtained by Mukohara \cite{Mu24} for quasi-product actions with simple fixed point algebras. 
She also showed that the fixed point inclusions for such actions are C$^*$-irreducible in the sense of 
R\o rdam \cite{R23}. 
This also proves the importance of quasi-product actions.  

One of the purposes of this paper is to prove the following result, which shows that the class of quasi-product actions 
is broad and relatively easy to recognize: 

\begin{theorem}\label{main1} Let $\alpha$ be a faithful action of a second countable compact group $G$ on a separable C$^*$-algebra $A$. 
We assume that $\alpha$ is minimal in the sense that the relative commutant ${A^\alpha}'\cap M(A)$ is trivial, where 
$A^\alpha$ is the fixed point algebra of $\alpha$ and $M(A)$ is the multiplier algebra of $A$. 
Assume further that $A^\alpha$ is simple. 
Then $\alpha$ is quasi-product. 
(Note that $A$ is necessarily simple thanks to \cite[Lemma 24]{WII}, \cite[Theorem 1,(3)]{BEK93}.)
\end{theorem}

When $G$ is abelian, the statement of the theorem is reduced to the proper outerness of the dual action 
(see \cite[Theorem 1]{BEEK89}), which in turn follows from Kishimoto's result \cite[Lemma 1.1]{K81} for automorphisms. 
In the general case, the dual action consists of endomorphisms whose proper outerness is technically much subtler as was 
observed in \cite{BEK93} and \cite{I02}. 
In fact, the author challenged the problem of generalizing Kisimoto's result to irreducible endomorphisms 
of finite index, and ended up with a result under the assumption of finite depth condition \cite[Theorem 7.5]{I02}, 
which implies the statement of Theorem \ref{main1} for profinite $G$.   
This work is to challenge the problem again after more than 20 years.  

Classification of group actions on Kirchberg algebras is a growing subject and a recent ground breaking result of 
Gabe-Szab\'o \cite{GS22-2} established an equivariant version of the Kirchberg-Phillips classification theorem \cite{Ph00} 
(see also \cite{GS22-1}, \cite{I10}, \cite{IMI}, \cite{IMII}, \cite{Me21}, \cite{OS21}, \cite{Sz21} for related results). 
They showed that amenable isometrically shift-absorbing actions of second countable locally compact groups on Kirchberg algebras 
are completely classified by equivariant $KK$-equivalence.  
An action of a locally compact group $G$ on a C$^*$-algebra $A$ is said to be isometrically shift-absorbing 
if the quasi-free action on the Cuntz algebra $\cO_\infty$ arising from the infinite direct sum of the regular representation of 
$G$ equivariantly embeds in the central sequence algebra of $A$. 
For a countable discrete group, this condition is equivalent to the outerness of the action (see \cite{IM2010}). 
However, for other classes of groups, it is a non-trivial task to characterize when this property holds. 
In fact, for the real numbers $G=\R$, it is equivalent to the Rokhlin property (see \cite[Corollary 6.15]{GS22-2}, \cite{Sz21}).  

When $G$ is compact, every isometrically shift-absorbing action is quasi-product by \cite[Theorem 1,(8)]{BEK93}, 
and Mukohara \cite{Mu24} showed that it is a minimal action having the fixed point algebra purely infinite and simple. 
We show the converse when the fixed point algebra is a Kirchberg algebra, as an application of Theorem \ref{main1}. 

\begin{theorem}\label{main2} Let $G$ be a second countable compact group, and let $\alpha$ be a faithful minimal action of 
$G$ on a separable C$^*$-algebra $A$ whose fixed point algebra is a Kirchberg algebra. 
Then $\alpha$ is isometrically shift-absorbing. 
(Note that $A$ is necessarily a Kirchbeg algebra thanks to \cite[Proposition 3]{DLRZ02}.)
\end{theorem}

This paper is organized as follows. 
In Section 2, we summarize the basics of properly outer endomorphisms, 
inclusions of simple C$^*$-algebras of finite indices, and quasi-product actions of compact groups.   
A common key notion in the above three subjects is, what we call, the property (BEK) for an inclusion of prime 
C$^*$-algebras introduced in Bratteli-Elliott-Kishimoto \cite{BEK93}. 

In Section 3, we show that every irreducible inclusion of purely infinite simple C$^*$-algebras of finite index has 
the property (BEK). 
For the proof we crucially use  an ultraproduct technique and the fact that every irreducible inclusion  
of simple C$^*$-algebras of finite index is C$^*$-irreducible in the sense of R\o rdam.  
As a consequence, we prove the proper outerness of every irreducible proper endomorphism of finite index in the case of purely infinite simple C$^*$-algebras. 
This immediately implies Theorem \ref{main1} in the case of purely infinite $A^\alpha$. 
The general case is reduced to this case by a tensor product trick in Section 4, but the reduction argument 
is rather complicated. 
In fact, we use ergodic actions of $G$, which appear for a tensor categorical reason (see Remark \ref{MC}).  

In Section 5 we show an equivariant version of the completely positive approximation property (CPAP) of a nuclear 
C$^*$-algebra with a compact group action, as preparation for the proof of Theorem \ref{main2}. 
Theorem \ref{Fejer} is of interest in its own right as a Fej\'er type approximation of a general compact group. 
We prove Theorem \ref{main2} in Section 6 by showing an equivariant version of Kirchberg's dilation theorem for 
unital completely positive (ucp) maps of Kirchberg algebras. 

In Section 7, we treat quasi-free actions of compact groups on the Cuntz algebras 
as applications of our main results.  
\section{Preliminaries}
\subsection{Notation}
We use the following notation throughout the paper. 
For a C$^*$-algebra $A$, we denote by $U(A)$ the unitary group of $A$.  
For $u\in U(A)$, we denote by $\Ad u$ the automorphism of $A$ defined by $\Ad u(x)=uxu^*$ for $x\in A$. 
An automorphism of $A$ is called inner if it is of the form $\Ad u$, and otherwise it is called outer.  
We denote by $A_1$ the unit ball of $A$, and by $A_+$ the set of positive elements in $A$. 
We denote by $M(A)$ the multiplier algebra of $A$. 
For two C$^*$-algebras $A$, $B$, we denote by $A\otimes B$ the minimal tensor product. 
An action $\alpha$ of a topological group $G$ on a C$^*$-algebra $A$ is a continuous homomorphism from $G$ into 
the automorphism group $\Aut(A)$ of $A$ equipped with the point norm topology.    
We denote the fixed point algebra of $\alpha$ by $A^\alpha$, or $A^G$ when there is no possibility of confusion.   

An inclusion of C$^*$-algebras $A\supset B$ is called irreducible if the relative commutant $M(A)\cap B'$ is trivial. 
We denote by $\End(A)$ the set of endomorphisms of $A$, and call $\rho\in \End(A)$ irreducible if the inclusion $A\supset \rho(A)$ is irreducible.  

For a C$^*$-algebra $A$ and a free ultra-filter $\omega\in \beta\N\setminus\N$, we define
\[A^\infty=\ell^\infty(\N,A)/c_0(\N,A),\]   
\[A^\omega=\ell^\infty(\N,A)/c_\omega(\N,A),\] 
\[ c_\omega(A)=\{(x_n)\in \ell^\infty(\N,A);\; \lim_{n\to\omega}\|x_n\|=0\}.\]
We treat $A$ as a subalgebra of $A^\infty$ and $A^\omega$ as usual. 

We denote by $\M_n(\C)$ the $n$ by $n$ matrix algebra. 
We denote by $\K(H)$ the set of compact operators on a Hilbert space $H$, 
and denote $\K=\K(\ell^2)$. 
We denote $\T=\{z\in \C;\; |z|=1\}$. 
We use the capital Greek letters $\Pi$, $\Sigma$, etc. for representations of C$^*$-algebras, and reserve $\pi$, $\sigma$, etc. 
for representations of compact groups. 
\subsection{Properly outer endomorphisms}
\begin{definition} Let $A$ be a C$^*$-algebra. We say that $\rho\in \End(A)$ is \textit{properly outer} if the following holds:  
for any $a\in A$ and non-zero hereditary C$^*$-subalgebra $C$ of $A$,   
\[\inf\{\|ca\rho(c)\|;\; c\in C_+,\; \|c\|=1\}=0.\]
\end{definition}

Kishimoto \cite[Lemma 1.1]{K81} showed that if $A$ is simple, any outer automorphism of $A$ is properly outer. 
It was observed in \cite[p.322]{BEK93} that the same argument shows 

\begin{lemma} \label{containment} 
Let $A$ be a separable simple C$^*$-algebra, and assume that $\rho\in \End(A)$ is not properly outer. 
Then for any irreducible representation $\Pi$ of $A$, the composition $\Pi\circ \rho$ contains $\Pi$ as a subrepresentation.  
\end{lemma}

The following theorem of Bratteli-Elliott-Kishimoto is one of the main technical results 
in their analysis of quasi-product actions of compact groups.  

\begin{theorem}[{\cite[Theorem 3.1]{BEK93}}] \label{BEK} 
Let $A\supset B$ be an inclusion of separable C$^*$-algebras. 
Then the following conditions are equivalent. 
\begin{itemize}
\item[$(1)$] For any $x,y\in A$, 
\[\sup_{b\in B_1}\|xby\|=\|x\|\|y\|.\]
\item[$(2)$] There exists $\delta>0$ such that for any $x,y\in A$, 
\[\sup_{b\in B_1}\|xby\|\geq \delta \|x\|\|y\|.\]
\item[$(3)$] There exists an faithful irreducible representation $\Pi$ of $A$ whose restriction to $B$ is irreducible. 
\end{itemize}
\end{theorem}

\begin{definition}
We call the condition of $A\supset B$ in Theorem \ref{BEK} the property (BEK).
\end{definition} 

\begin{remark} The property (BEK) forces the inclusion $A\supset B$ to be irreducible. 
\end{remark}

The following lemma is shown in \cite[p.322]{BEK93} when there exists $n\geq 2$ such that $\rho^n$ is reducible. 
For the sake of completeness, we give a proof in full generality.  

\begin{lemma}\label{BEKPO} 
Let $A$ be a separable simple C$^*$-algebra and let $\rho\in \End(A)$. 
We assume that the image of $\rho$ is a proper irreducible subalgebra of $A$ and the inclusion $A\supset \rho(A)$ has the property (BEK). 
Then $\rho$ is properly outer.  
\end{lemma}

\begin{proof} Assume on the contrary that $\rho$ is not properly outer. 
Since $A\supset \rho(A)$ has the property (BEK), there exists an irreducible representation $(\Pi,H)$ of $A$ whose 
restriction to $\rho(A)$ is irreducible, which means that $\Pi\circ \rho$ is irreducible. 
Since $\rho$ is not properly outer, the composition $\Pi\circ \rho$ contains $\Pi$, and so $\Pi\circ \rho$ is 
unitarily equivalent to $\Pi$, and there exits a unitary $u\in B(H)$, unique up to scalar multiple, 
satisfying $\Pi\circ \rho=\Ad u\circ \Pi$. 
If there exists $n\geq 2$ such that $\rho^n$ is reducible, on one hand $\Pi\circ \rho^n$ is unitarily equivalent 
to $\Pi$ and on the other hand it is reducible, which is a contradiction. 
Thus $\rho^n$ is irreducible for all $n\in \N$. 

Proof of \cite[Lemma 1.1]{K81} shows that there exist $a\in A\setminus \{0\}$, $\delta>0$, and a non-zero hereditary 
C$^*$-subalgebra $C\subset A$ such that for any $c\in C$, 
\begin{equation}\label{Kis1}
\Pi(c^*)(\Pi(a)u+u^*\Pi(a^*))\Pi(c)\geq \delta \Pi(c^*c).
\end{equation}
For non-negative integers $n$, we set $A_n=u^{*n}\Pi(A)u^n$. 
Then $\{A_n\}_{n=0}^\infty$ is a strictly increasing sequence of simple C$^*$-algebras. 
Let $B$ be the norm closure of its union, which is simple. 
We introduce $\theta\in \Aut(B)$ by the restriction of $\Ad u$ to $B$. 
By construction, we have $\theta(A_n)=A_{n-1}$ and $\theta(\Pi(x))=\Pi(\rho(x))$ for all $x\in A$. 

We claim that $\theta^n$ is outer for all $n\in \N$. 
Assume on the contrary that there exists $n\in \N$ with $\theta^n$ inner. 
Then there exists $v\in U(B)$ satisfying $\theta^n=\Ad v$. 
Since $\theta^n(\Pi(x))=\Ad u^n(\Pi(x))$, we see that $v^*u^n\in \Pi(A)'=\C$, and 
there exists $c\in \T$ satisfying $v=cu^n$. 
For any $\varepsilon>0$, there exist $m\in \N$ and $w\in U(A_m)$ satisfying $\|v-w\|<\varepsilon$. 
Then $\Ad u^m(w)\in \Pi(A)$ and 
\[\|cu^n-\Ad u^m(w)\|=\|\Ad u^m(cu^n-w)\|<\varepsilon.\] 
Since $\varepsilon>0$ is arbitrary, this means $u^n \in \Pi(A)$, and $\rho^n$ is an inner automorphism of $A$, 
which contradicts $A\supsetneqq  \rho(A)$. 
Thus the claim is shown.     

By the claim, the crossed product $B\rtimes_\theta\Z$ is simple, and it is canonically isomorphic to 
the C$^*$-algebra $C^*(B\cup\{u\})$ generated by $B$ and $u$. 
Thus there exists a conditional expectation $E:C^*(B\cup\{u\})\to B$ satisfying $E(u)=0$. 
This contradicts Eq.(\ref{Kis1}), and we get the statement. 
\end{proof}

\subsection{Finite index inclusions of simple C$^*$-algebras} 
Our basic reference for inclusions of simple C$^*$-algebras is \cite{I02}. 

Let $A\supset B$ be an inclusion of C$^*$-algebras with a conditional expectation $E:A\to B$. 
Pimsner-Popa \cite{PP86} defined the index of $E$, denoted by  $\Ind E$, 
by the best constant $\lambda \geq 1$ such that the map $E-\lambda^{-1}\id:A\to A$ is positive 
(note that \cite[Definition 2.1]{I02} contains typographic errors). 
When there is no such a constant $\lambda$, we set $\Ind E=\infty$. 
When $B$ has no non-zero finite dimensional representation, we can replace positivity by complete positivity 
in the definition (see \cite[Lemma 2.2]{I02}).  

Watatani \cite{Wat90} introduced an alternative definition of the index of $E$ in terms of a quasi-basis, 
a generalization of the Pimsner-Popa basis. 
His definition is particularly well-behaved when we discuss the C$^*$-basic construction, which we will do now. 
The two definitions coincide when $A$ and $B$ are infinite dimensional simple C$^*$-algebras 
\cite[Theorem 3.2, Corollary 3.7]{I02}, which we always assume in what follows.

For a given conditional expectation $E:A\to B$ of finite index, we can introduce a Hilbert $B$-module $\cE_E$ as follows. 
We let $\cE_E=A$ as a right-$B$ module, and introduce $B$-valued inner product by $\inpr{x}{y}_E=E(x^*y)$. 
The Pimsner-Popa inequality assures that $\cE_E$ is already complete. 
When $x\in A$ is regarded as an element in $\cE_E$, we often denote $\eta_E(x)$ to avoid possible confusion.  
We denote by $\mathbb{L}_B(\cE_E)$ the C$^*$-algebra of adjointable $B$-module maps on $\cE_E$. 
We regard $A$ as a C$^*$-subalgebra of $\mathbb{L}_B(\cE_E)$ through the left multiplication. 
The Jones projection $e_E\in \mathbb{L}_B(\cE_E)$ is defined by $e_E\eta_E(x)=\eta_E(E(x))$ for $x\in A$, 
which belongs to the commutant of $B$.  
Then we have the relation $e_Exe_E=E(x)e_E$ for all $x\in A$.  
The C$^*$-basic construction of $A\supset B$ is the norm closure of the linear span of $Ae_EA$, which coincides with 
the set of ``compact operators" $\K_B(\cE_E)$.  
Under the assumption of $\Ind E<\infty$ and the simplicity of $B$, we always have $\K_B(\cE_E)\supset A$ 
and there exists a conditional expectation $E_1:\K_B(\cE_E)\to A$ given by $E_1(e_E)=1/\Ind E$, which is called the dual 
conditional expectation of $E$ (see\cite[Corollary 3.4]{I02}). 
We have $\Ind E_1=\Ind E$. 

When there exists a conditional expectation $E:A\to B$ of finite index, 
there exists a unique conditional expectation $E_0:A\to B$ satisfying $\Ind E_0\leq \Ind E$ for all 
faithful $E$ (see \cite{H88}, \cite{Wat90}). 
We call $E_0$ the minimal conditional expectation of $A\supset B$. 
We denote $[A:B]_0=\Ind E_0$ and call it the minimum index of $A\supset B$. 
If moreover $A\supset B$ is irreducible, there is only one conditional expectation. 

The following result is a consequence of the second dual approach developed in \cite{I02}, 
which will be used in one of our main technical results (see \cite[Theorem 3.3]{I02}). 

\begin{lemma}\label{intermediate} 
Let $A\supset B$ be an irreducible inclusion of simple C$^*$-algebras with 
a conditional expectation $E:A\to B$ of finite index. 
Then every intermediate C$^*$-subalgebra between $A$ and $B$ is simple, that is, the inclusion $A\supset B$ is 
C$^*$-irreducible in the sense of R{\o}rdam \cite[Definition 3.1]{R23}. 
\end{lemma}

Next we give a brief account of sector theory for C$^*$-algebras developed in \cite[Section 4]{I02}. 
For two simple $C^*$-algebras $A$ and $B$, we denote by $\Mor(B,A)_0$ the set of homomorphisms from $B$ into $A$ 
whose image has a finite index. 
When $A$ and $B$ are non-unital, every $\rho\in \Mor(B,A)_0$ uniquely extends to a strictly continuous homomorphism 
from $M(B)$ to $M(A)$, which we denote by the same symbol $\rho$.  
For $\rho\in \Mor(B,A)_0$, we denote by $E_\rho$ the minimal conditional expectation $E_\rho:A\to \rho(B)$, and denote 
$d(\rho)=[A:\rho(B)]_0^{1/2}$, which is called the statistical dimension of $\rho$. 
We say that two homomorphisms $\rho, \sigma\in \Mor(B,A)$ are equivalent if there exists a unitary $u\in M(A)$ satisfying 
$\rho=\Ad u\circ \sigma$. 

For $\rho,\sigma\in \Mor(B,A)_0$, we denote by $(\rho,\sigma)$ the intertwiner space 
$$(\rho,\sigma)=\{v\in M(A);\; \forall x\in B,\; v\rho(x)=\sigma(x)v\},$$
which is always finite dimensional. 
When $\rho$ is irreducible, it is a Hilbert space with $\inpr{v}{w}=w^*v$. 

We mainly work on the following two classes of C$^*$-algebras: 
\begin{itemize}
\item $\cC_1$: the class of simple stable $\sigma$-unital C$^*$-algebras.
\item $\cC_2$: the class of unital purely infinite simple C$^*$-algebras in the Cuntz standard form, 
that is $[1_A]_0=0$ in $K_0(A)$. 
\end{itemize}

Assume that $A$ and $B$ are C$^*$-algebras belonging to either $\cC_1$ or $\cC_2$. 
Then every $\rho\in \Mor(B,A)_0$ has its conjugate $\brho\in \Mor(A,B)_0$, uniquely determined up to equivalence  
and characterized by the following property: there exist isometries $R_\rho\in (\id_B,\brho\rho)$ and 
$\bR_\rho \in (\id_A,\rho\brho)$ satisfying 
\[{\bR_\rho}^*\rho(R_\rho)=\frac{1}{d(\rho)},\quad {R_\rho}^*\brho(\bR_\rho)=\frac{1}{d(\rho)}. \]
The inclusion $B\supset \brho(A)$ is isomorphic to the dual inclusion $\K_{\rho(B)}(\cE_{E_\rho})\supset A$ of $A\supset \rho(B)$ 
(see \cite[Lemma 4.4]{I02}). 

For $\rho,\sigma\in \Mor(B,A)_0$, their direct sum $\rho\oplus \sigma\in \Mor(B,A)_0$, which is uniquely determined up to 
unitary equivalence, is defined as follows: 
we choose two isometries $s_1,s_2\in M(A)$ satisfying the Cuntz algebra $\cO_2$ relation $s_1s_1^*+s_2s_2^*=1$ and let 
\[(\rho\oplus \sigma)(x)=s_1\rho(x)s_1^*+s_2\sigma(x)s_2^*.\] 
With these operations together with composition as monoidal product,  
the two sets $\End(A)_0$ and $\End(B)_0$ are rigid C$^*$-tensor categories and $\Mor(A,B)_0$ and $\Mor(B,A)_0$ are 
their module categories. 
In particular, the Frobenius-reciprocity holds (see \cite{I98}).   

One of advantages to work in the classes $\cC_i$, $i=1,2$, is that we have the following crossed product type decomposition. 
Let $A\supset B$ be an irreducible inclusion of simple C$^*$-algebras belonging to either $\cC_1$ or $\cC_2$ 
with a conditional expectation $E:A\to B$ of finite index. 
Let $\iota :B\hookrightarrow A$ be the inclusion map and let 
\[[\biota\iota]=\bigoplus_{\xi\in \Xi}n_\xi[\rho_\xi]\]
be the irreducible decomposition, where $n_\xi$ is the multiplicity of $\rho_\xi\in \End(B)_0$. 
We may and do assume $0\in \Xi$ and $\rho_0=\id_B$. 
Since $\iota$ is irreducible, we have $n_0=1$. 
By the Frobenius reciprocity, 
\[\dim(\iota,\iota\rho_\xi)=\dim(\biota\iota,\rho_\xi)=n_\xi.\]
Let $\{V(\xi)\}_{i=1}^{n_\xi}$ be an orthonormal basis of $(\iota,\iota\rho_\xi)$.  
Then every $x\in A$ is uniquely decomposed as 
\[x=\sum_{\xi\in \Xi}\sum_{i=1}^{n_\xi}x(\xi)_iV(\xi)_i,\]
\[x(\xi)_i=d(\rho_\xi)E(xV(\xi)_i^*)\in B,\]
(see \cite[p.124]{I02}).  
Note that $x(0)_1=E(x)$, and  
\[b(x-E(x))b=\sum_{\xi\in \Xi\setminus\{0\}}\sum_{i=1}^{n_\xi}bx(\xi)_i\rho_\xi(b)V(\xi)_i,\quad b\in B.\] 

For two representations $(\Pi_1,H_1)$, $(\Pi_2,H_2)$ of a C$^*$-algebra $A$, we denote by $\Hom_A(\Pi_1,\Pi_2)$ the set of 
intertwiners 
\[\{T\in \B(H_1,H_2);\;\forall x\in A,\; T\Pi_1(x)=\Pi_2(x)T\}.\]

\begin{lemma}\label{finite}
\begin{itemize}
\item[$(1)$] Let $A\supset B$ be an inclusions of simple C$^*$-algebras of finite index, and let $\Pi$ be 
an irreducible representation of $A$.  
Then the restriction of $\Pi$ to $B$ is a finite direct sum of irreducible representations. 
\item[$(2)$] Let $A$ and $B$ be C$^*$-algebras belonging to $\cC_1$ or $\cC_2$, and let $\rho\in \Mor(B,A)_0$. 
Let $\Phi$ and $\Psi$ be irreducible representations of $A$ and $B$ respectively. 
Then 
\[\dim \Hom_A(\Phi\circ\rho,\Psi)=\dim \Hom_A(\Phi,\Psi\circ \brho).\]
\end{itemize}
\end{lemma}

\begin{proof} The statements follow from \cite[Lemma 5.1, Lemma 5.2]{I02}. 
\end{proof}

\subsection{Quasi-product actions}
For a compact group $G$, we denote by $\cU(G)$ the category of finite dimensional unitary representations of $G$. 
We choose and fix a transversal $\hG$ of the set of equivalence classes of irreducible unitary representations of $G$. 
For $\pi\in \hG$, we denote $\chi_\pi(g)=\Tr\pi(g)$ and $d(\pi)=\dim \pi$. 
We choose and fix an orthonormal basis $\{\xi(\pi)\}_{i=1}^{d(\pi)}$ of the representation space $H_\pi$ of $\pi$, 
and identify $\pi(g)$ with its matrix representation $(\pi_{ij}(g))_{ij}$. 
For $\pi,\sigma,\mu\in \hG$, we let $N_{\pi,\sigma}^\mu=\dim\Hom_G(\mu, \pi\otimes \sigma)$. 

Let $\alpha$ be an action of $G$ on a C$^*$-algebra $A$. 
For $\pi\in \hG$, we define $P_\pi:A\to A$ by 
\[P_\pi(x)=d(\pi)\int_G\overline{\chi_\pi(g)}\alpha_g(x)dg,\]
where $dg$ is the normalized Haar measure. 
Then we have $P_\pi\circ P_\sigma=\delta_{\pi,\sigma}P_\pi$ for $\pi,\sigma\in \hG$.  
The image of $P_\pi$ is called the spectral subspace corresponding to $\pi\in \hG$, and we denote it by $A^\alpha(\pi)$.  
The Peter-Weyl theorem together with the Hahn-Banach theorem implies that the linear span of $\cup_{\pi\in \hG}A^\alpha(\pi)$ 
is dense in $A$. 
Or alternatively, we can see it more directly from a Fej\'er type approximation (see Theorem \ref{Fejer}). 
Let 
\[A^\alpha_1(\pi)=\{x=(x_1,x_2,\ldots,x_{d(\pi)})\in A^{d(\pi)};\; \alpha_g(x)=x\pi(g)\}.\]
Then $A^\alpha(\pi)\neq \{0\}$ if and only if $A^\alpha_1(\pi)\neq \{0\}$. 

Let $E_\alpha=P_1$, which is a conditional expectation from $A$ to the fixed point algebra $A^\alpha$. 
From the compactness of $G$, we get the following well-known lemma:  

\begin{lemma}\label{appunit}
If $\{u_\lambda\}_{\lambda\in \Lambda}$ is an approximate units of $A$, 
so is $\{E_\alpha(u_\lambda)\}_{\lambda\in \Lambda}$. 
\end{lemma}

The lemma implies $\overline{A^\alpha A}^{\|\cdot\|}=A$, and we have the inclusion relation $M(A^\alpha)\subset M(A)$. 
We denote the extension of $\alpha_g$ to $M(A)$ by the same symbol $\alpha_g$. 
Then the map $g\mapsto \alpha_g(T)$ is is continuous in the strict topology for all $T\in M(A)$.   
We have $M(A^\alpha)=M(A)^\alpha$ (see, for example, \cite[Lemma 2]{Pel24}). 

Bratteli-Elliott-Kishimoto \cite[Theorem 1]{BEK93} showed that 10 conditions on a compact group action 
are mutually equivalent, which are the defining conditions of a quasi-product action. 
We name some of them under the assumption of Theorem \ref{main1} now. 
Recall that $\alpha$ is said to be minimal if the inclusion $A\supset A^\alpha$ is irreducible, which is a necessary condition 
for $\alpha$ to be quasi-product because of the condition (1) below. 
Note that the crossed product $A\rtimes_\alpha G$ is automatically simple under the assumption of Theorem 1.1 
(see \cite[Proposition A]{Mu24}). 

\begin{theorem}[Quasi-product action]\label{qp1} 
Let $G$ be a second countable compact group, and let $\alpha$ be a faithful minimal action of $G$ 
on a separable C$^*$-algebra $A$ whose fixed point algebra $A^\alpha$ is simple. 
Then the following conditions are equivalent:
\begin{itemize}
\item[$(1)$] The inclusion $A\supset A^\alpha$ has the property (BEK). 
\item[$(2)$] There exists an $\alpha$-invariant pure state of $A$. 
\item[$(3)$] For each $\pi\in \hG$, there exists a sequence $\{y_n\}_{n=1}^\infty$ in $A^\alpha_1(\pi)$ such that  
\begin{itemize}
\item[$(\mathrm{i})$] $\|y_{n1}\|=1$ for all $n\in \N$,  
\item[$(\mathrm{ii})$] $\{y_{n1}\}_{n=1}^\infty$ is a central sequence in $A$, 
\item[$(\mathrm{iii})$] $\limsup_n\|ay_{n1}\|\geq \|a\|/d(\pi)$ for all $a\in A$. 
\end{itemize}
\item[$(4)$] The dual endomorphisms $\halpha_\pi$, $\pi\in \hG\setminus\{1\}$, of the  stabilized crossed product 
$(A\rtimes_\alpha G)\otimes \K$ are properly outer.   
\end{itemize}
\end{theorem}

\begin{remark} When the above equivalent conditions hold, the C$^*$-algebra $A$ is simple too, which follows from Landstad's 
observation \cite[Lemma 24]{WII} because (1) implies that $A$ is prime. 
We can also use (4) to show that $A\otimes \K$ is simple as in the case of a crossed product by a 
discrete group (see \cite[Theorem 3.1]{K81}). 
\end{remark}

The condition (4) is the most tractable for our purpose. 
On the other hand, the definition of $\halpha_\pi$ in \cite{BEK93} is rather tedious, and we reformulate it, 
following the original formulation by Roberts \cite{Ro75}, 
in the special case where $A^\alpha$ belongs to either $\cC_1$ or $\cC_2$. 

By a Hilbert space $\cH$ in $M(A)$, we mean a finite dimensional subspace of $M(A)$ such that it is a Hilbert space 
with respect to the inner product given by $\inpr{V}{W}=W^*V$ for $V,W\in \cH$. 
For a Hilbert space $\cH$ in $M(A)$, we choose an orthonormal basis $\{\psi(\cH)_i\}_i$ of $\cH$ and define the support 
$s(\cH)$ of $\cH$ by 
$$\sum_i\psi(\cH)_i\psi(\cH)_i^*,$$
which does not depend on the choice of $\{\psi(\cH)_i\}_i$. 
For two Hilbert spaces $\cH_1,\cH_2$ in $M(A)$, we denote by $\cH_1\cH_2$ and $\cH_1\cH_2^*$ the linear spans of 
$\{VW;\;V\in \cH_1,\; W\in \cH_2 \}$ and $\{VW^*;\;V\in \cH_1,\; W\in \cH_2 \}$ respectively. 
Then $\cH_1\cH_2$ is identified with the tensor product $\cH_1\otimes \cH_2$, and $\cH_1\cH_2^*$ is identified with 
$\B(\cH_2,\cH_1)$ through left multiplication. 
Occasionally, we consider infinite dimensional Hilbert spaces in $M(A)$, and in that case we take the norm closure 
to define $\cH_1\cH_2$ and $\cH_1\cH_2^*$. 
Then $\cH_1\cH_2$ is a Hilbert space again and $\cH_1\cH_2^*$ is isometric to the space of compact operators $\K(\cH_2,\cH_1)$.

If a Hilbert space $\cH$ in $M(A)$ is globally invariant under $\alpha$, it carries a unitary representation of $G$, 
and hence $\cH$ is considered as an object in the category $\cU(G)$. 
We denote by $\cU(G,\alpha)$ the set of globally $\alpha$-invariant Hilbert spaces in $M(A)$ with support 1. 
For $\cH_1,\cH_2\in \cU(G,\alpha)$, we define the space of morphisms $(\cH_1,\cH_2)_G$ from $\cH_1$ to $\cH_2$ 
in the category $\cU(G,\alpha)$ by 
\[(\cH_1,\cH_2)_G:=\cH_2\cH_1^*\cap M(A^\alpha). \]
In this way, we get a tensor subcategory $\cU(G,\alpha)$ of $\cU(G)$.

For $\cH\in \cU(G,\alpha)$, we define $\rho_\cH\in \End(A)_0$ by 
\[\rho_\cH(x)=\sum_i\psi(\cH)_ix\psi(\cH)_i^*,\]
which does not depend on the choice of the orthonormal basis, and hence $\rho_\cH$ commutes with $\alpha_g$. 
By construction, we have $[\rho_\cH]=\dim \cH[\id_A]$, and $d(\rho_\cH)=\dim \cH$. 
We define $\calpha_\cH\in \End(A^\alpha)_0$ to be the restriction of $\rho_\cH$ to $A^\alpha$. 
Note that $Vx=\calpha_\cH(x)V$ holds for all $x\in A^\alpha$ and $V\in \cH$. 
Since $(\rho_{\cH_1},\rho_{\cH_2})=\cH_2\cH_1^*$, we have $(\cH_1,\cH_2)_G\subset (\calpha_{\cH_1},\calpha_{\cH_2})$,  
and the irreducibility of $A\supset A^\alpha$ implies the equality of the two sets. 
Since $\rho_{\cH_1}\circ \rho_{\cH_2}=\rho_{\cH_1\cH_2}$, we have $\calpha_{\cH_1\cH_2}=\calpha_{\cH_1}\circ \calpha_{\cH_2}$.  
Setting $\calpha_T=T$ for $T\in (\cH_1,\cH_2)_G$, we get a tensor functor $\calpha: \cU(G,\alpha)\to \End(A^\alpha)_0$, 
which is in fact an embedding of $\cU(G,\alpha)$ into $\End(A^\alpha)_0$.  

Using the assumption that $A^\alpha$ belongs to either $\cC_1$ or $\cC_2$, we can show that $\cU(G,\alpha)$ has a direct sum, 
up to equivalence, and that every object in $\cU(G,\alpha)$ is a direct sum of simple objects. 
Indeed, for $\cH_1,\cH_2\in \cU(G,\alpha)$, we choose isometries $s_1,s_2\in M(A^\alpha)$ satisfying the $\cO_2$ relation 
$s_1s_1^*+s_2s_2^*=1$, and set $\cH_1\oplus \cH_2=s_1\cH_1+s_2\cH_2\in \cU(G,\alpha)$. 
Then we have 
\[\calpha_{\cH_1\oplus \cH_2}(x)=s_1\calpha_{\cH_1}(x)s_1^*+s_2\calpha_{\cH_2}(x)s_2^*,\]
which is compatible with the definition of a direct sum in $\End(A^\alpha)_0$. 
Let $\cH\in \cU(G,\alpha)$ and let $\{p_i\}_{i=1}^n$ be a set of minimal projections in $(\cH,\cH)_G=(\calpha_\cH,\calpha_\cH)$ 
whose summation is 1. 
Then \cite[Lemma 4.1]{I02} shows that $p_i$ is equivalent to 1 in $M(A^\alpha)$, and there exist isometries $V_i\in M(A^\alpha)$ 
satisfying $V_iV_i^*=p_i$. 
Now $\cH_i:=V_i^*\cH\in \cU(G,\alpha)$, and $\cH$ is equivalent to the direct sum of $\cH_i$, $i=1,2,\ldots,n$. 

\begin{lemma}\label{equivalence1} Under the assumption of Theorem \ref{main1}, assume that $A^\alpha$ belongs to either $\cC_1$ or $\cC_2$. 
Then for every representation $\sigma \in \cU(G)$, there exists $\cH_\sigma\in \cU(G,\alpha)$ equivalent to $\sigma$. 
\end{lemma}

\begin{proof} This essentially follows from \cite[Lemma III 3.4]{AHKT77}. 
\end{proof}

\begin{remark}\label{equivalence2} If we only assume that $A^\alpha$ is unital simple purely infinite in the above lemma, we can still have 
a globally $\alpha$-invariant Hilbert space equivalent to $\sigma$, not necessarily with support 1. 
Indeed, choosing a non-zero projection $e\in A^\alpha$ with $[e]_0=0$ in $K_0(A^\alpha)$ and applying the lemma to $eAe$, 
we get $\cH_\sigma\in \cU(G,\alpha|_{eAe})$ equivalent to $\sigma$. 
We choose an isometry $V \in A^\alpha$ with $VV^*\leq e$. 
Then $\cH_\sigma V$ is the desired Hilbert space.   
\end{remark}

We call the functor $\calpha$ the pre-dual action of $\alpha$. 
For each $\pi\in \hG$, we fix $\cH_\pi\in \cU(G,\alpha)$ equivalent to $\pi$, and denote 
$\calpha_\pi=\calpha_{\cH_\pi}$ for simplicity. 
We arrange the orthonormal basis $\{\psi(\cH_\pi)_i\}_i$ so that it is consistent with the orthonormal basis 
$\{\xi(\pi)_i\}_i$ of $H_\pi$ we have already chosen, and denote $\psi(\pi)_i=\psi(\cH_\pi)_i$ for simplicity. 

Let $\bpi$ be the complex conjugate representation of $\pi\in \hG$.  
We choose an orthonormal basis $\{\xi(\bpi)_i\}_i$ of the representation space of $\bpi$ 
so that $\bpi_{ij}(g)=\overline{\pi_{ij}(g)}$ holds. 
When $\pi$ and $\bpi$ are inequivalent, we can arrange $\bpi$ to be a member of $\hG$. 
When $\pi$ and $\bpi$ are equivalent, we have two specially chosen orthonormal bases $\{\psi(\pi)_i\}_i$ 
and $\{\psi(\bpi)_i\}_i$ of $\cH_\pi$ corresponding to $\{\xi(\pi)_i\}_i$ and $\{\xi(\bpi)_i\}_i$ respectively . 

For $\pi\in \hG$, we set 
\[R_\pi=\frac{1}{\sqrt{d(\pi)}}\sum_{i=1}^{d(\pi)}\psi(\bpi)_i\psi(\pi)_i\in (\id,\calpha_{\bpi}\calpha_\pi),\] 
\[\bR_\pi=\frac{1}{\sqrt{d(\pi)}}\sum_{i=1}^{d(\pi)}\psi(\pi)_i\psi(\bpi)_i\in (\id,\calpha_{\pi}\calpha_{\bpi}),\]
which are isometries satisfying 
\[\bR_\pi^*\calpha_\pi(R_\pi)=R_\pi^*\calpha_{\bpi}(\bR_\pi)=\frac{1}{d(\pi)}.\]
This in particular shows that $d(\calpha_\pi)=d(\pi)$ as $\calpha_\pi$ is irreducible, 
and more generally shows $d(\calpha_\cH)=\dim \cH$ by the additivity of the statistical dimension.  

The standard left inverse $\phi_\pi:A^\alpha\to A^\alpha$ of $\calpha_\pi$ is given by 
\[\phi_\pi(x)=\frac{1}{d(\pi)}\sum_{i=1}^{d(\pi)}\psi(\pi)_i^*x\psi(\pi)_i,\quad x\in A^\alpha,\]
which is also expressed as $\phi_\pi(x)=R_\pi^*\calpha_{\bpi}(x)R_\pi$. 
Then the minimal conditional expectation for the inclusion $A^\alpha\supset \calpha_\pi(A^\alpha)$ can be expressed as 
$\calpha_\pi\circ \phi_\pi$. 

Now we can reformulate the (4) part of Theorem \ref{qp1} as follows:

\begin{lemma}\label{predual action} Under the assumption of Theorem \ref{main1}, assume that $A^\alpha$ belongs to either 
$\cC_1$ or $\cC_2$. 
Then $\alpha$ is quasi-product if and only if the endomorphisms $\calpha_\pi$, $\pi\in \hG\setminus\{1\}$, are properly outer. 
\end{lemma}

We denote by $A_0$ the linear span of $\cup_{\pi\in \hG}A^\alpha(\pi)$, which is a $G$-invariant dense $*$-subalgebra of $A$. 
A direct computation shows the following: 

\begin{lemma}\label{expansion}
Under the assumption of Theorem \ref{main1}, assume that $A^\alpha$ belongs to either $\cC_1$ or $\cC_2$. 
Then 
\begin{itemize}
\item[$(1)$] For any $x\in A^\alpha$ and $\pi\in \hG$, we have 
\[P_\pi(x)=d(\pi)\sum_{i=1}^{d(\pi)}E_\alpha(x\psi(\pi)_i^*)\psi(\pi)_i
=d(\pi)\sum_{i=1}^{d(\pi)}\psi(\bpi)_i^*E_\alpha(\psi(\bpi)_ix).\]
In particular, the linear span of $\cup_{\pi\in \hG}A^\alpha\cH_\pi$ is a dense $*$-subalgebra of $A$. 
\item[$(2)$] Every $x\in A_0$ is uniquely expanded as 
\[x=\sum_{\pi\in \hG}\sum_{i=1}^{d(\pi)}x(\pi)_i\psi(\pi)_i,\]
where $x(\pi)_i=d(\pi)E_\alpha(x\psi(\pi)_i^*)\in A^\alpha$. 
\end{itemize}
\end{lemma}

Since we need the Doplicher-Roberts reconstruction theorem \cite{DR89I}, we recall the permutation symmetry $\theta$ of 
the category $\cU(G,\alpha)$,   which is an assignment of a unitary $\theta(\cH_1,\cH_2)\in 
(\cH_1\cH_2,\cH_2\cH_1)_G$ 
to each $\cH_1.\cH_2\in \cU(G,\alpha)$ given by 
\[\theta(\cH_1,\cH_2)=\sum_{ij}\psi(\cH_2)_i\psi(\cH_1)_i\psi(\cH_2)_j^*\psi(\cH_1)_i^*.\]
Then we have 
\begin{equation}\label{P1}
\theta(\cH_1,\cH_2)^*=\theta(\cH_2,\cH_1),
\end{equation}
\begin{equation}\label{P2}
\theta(\cH_1\otimes \cH_2,\cH_3)=\theta(\cH_1,\cH_3)\rho_{\cH_1}(\theta(\cH_2,\cH_3)),
\end{equation}
\begin{equation}\label{P3}
\rho_{\cH_1}(T)=\theta(\cH_3,\cH_1)T\theta(\cH_2,\cH_1)^*,\quad \forall T\in (\cH_2,\cH_3)_G. 
\end{equation}

Before ending this section, we recall quasi-free actions on the Cuntz algebras.  
The Cuntz algebra $\cO_\infty$ is the universal C$^*$-algebra generated by a separable infinite dimensional 
Hilbert space $\cH$ in a C$^*$-algebra.  
By universality, every unitary representation of a group $G$ in $\B(\cH)$ induces its action on $\cO_\infty$, which is called 
the quasi-free action arising from the unitary representation. 
We define the Cunzt algebra $\cO_n$, $n=2,3,\ldots$, and quasi-free actions on them in a similar way with $\dim \cH=n$ and  
an extra condition that $\cH$ has support 1. 
\section{Finite index inclusions with the property (BEK)} 
The following is our first main technical result. 

\begin{theorem}\label{purely infinite} 
Let $A\supset B$ be an irreducible inclusion of separable unital purely infinite simple C$^*$-algebras with 
a conditional expectation $E:A\to B$ of finite index. 
Then the inclusion $A\supset B$ has the property (BEK). 
\end{theorem}

We need to prepare a few lemmata before proving the theorem. 
For two operators $S,T$ acting on the same Hilbert space, we denote $[S,T]=ST-TS$. 

\begin{lemma}\label{scalar} 
Let $B$ be a C$^*$-algebra, let $(\Phi,K)$ be an irreducible representation of $B$, and 
let $\delta>0$. 
Then if $T\in \B(K)_1$ satisfies 
\[\sup_{b\in B_1}\|[\Phi(b),T]\|\leq \delta,\]
there exists $\lambda\in \C$ satisfying $\|T-\lambda 1_K\|\leq 3\delta$. 
\end{lemma}

\begin{proof} By the Kaplansky density theorem, we have $\|[S,T]\|\leq \delta$ for all $S\in \B(K)_1$. 
We assume $\dim K=\infty$ as the finite dimensional case can be easily handled. 

We first claim that $\|T\xi-\inpr{T\xi}{\xi}\xi\|\leq\delta$ holds for every unit vector $\xi\in K$. 
The claim holds for an eigenvector of $T$.  
Assume that $\xi$ and $T\xi$ are linearly independent, and let $\lambda=\inpr{T\xi}{\xi}$ and 
$\eta=\|T\xi-\lambda\xi\|^{-1}\left(T\xi-\lambda\xi\right)$.
Then we can choose a unitary $u_{\pm}\in \B(K)$ satisfying $u_\pm\xi=\eta$ and $u_\pm\eta=\pm \xi$.  
We have 
\[[u_\pm,T]\xi=u_\pm T\xi-T\eta
 =u_\pm(\lambda\xi+\|T\xi-\lambda\xi\|\eta)-T\eta 
 =\lambda\eta-T\eta\pm \|T\xi-\lambda\xi\|\xi.\] 
Since $\|[u_\pm,T]\|\leq\delta$, we get $\|T\xi-\lambda\xi\|\leq\delta$, and the claim is shown. 

Next we claim that if $\xi$ and $\eta$ are mutually orthogonal unit vectors in $K$, we have 
$|\inpr{T\xi}{\xi}-\inpr{T\eta}{\eta}|\leq\delta$. 
We choose a unitary $u\in \B(K)$ satisfying 
$u\xi=\eta$ and $u\eta=\xi$. 
Note that we have $u^*\eta=\xi$. 
Then the claim follows from
\[\inpr{[u,T]\xi}{\eta}=\inpr{T\xi}{u^*\eta}-\inpr{Tu\xi}{\eta}=
\inpr{T\xi}{\xi}-\inpr{T\eta}{\eta}. 
\]

Finally, we fix a unit vector $\xi\in K$, and set $\lambda=\inpr{T\xi}{\xi}$. 
Let $\eta$ be an arbitrary unit vector in $K$. 
Choosing a unit vector $\zeta\in \{\xi,\eta\}^\perp$ and applying the second claim, we get 
$|\inpr{T\xi}{\xi}-\inpr{T\eta}{\eta}|\leq 2\delta$, 
which together with the first claim implies 
\[\|T\eta-\lambda\eta\|\leq \|T\eta-\inpr{T\eta}{\eta}\eta\|+\|\inpr{T\eta}{\eta}\eta-\lambda \eta\|\leq3\delta.\]
Thus the statement is shown. 
\end{proof}

\begin{lemma}\label{commutant} 
Let  $A\supset B$ be an inclusion of C$^*$-algebras with finite index, and let 
$(\Pi,H)$ be an irreducible representation of $A$. 
Then there exists a constant $C>0$ depending only on $\Pi$ satisfying the following property: 
whenever $T\in \B(H)_1$ satisfies 
\[\sup_{b\in B_1}\|[\Pi(b),T]\|\leq\delta,\]
there exists $T_1\in \Pi(B)'$ satisfying $\|T_1-T\|\leq C\delta$. 
\end{lemma}

\begin{proof} 
Thanks to Lemma \ref{finite}, the restriction $\Pi|_B$ of $\Pi$ to $B$ is a finite direct sum of irreducible representations. 
Let $(\Pi_k,H_k)$, $k=1,2,\ldots,m$, be the mutually disjoint irreducible components of 
the restriction $\Pi|_B$. 
Then we may assume
\[H=\bigoplus_{k=1}^m H_k\otimes \C^{n_k},\quad \Pi(b)=\bigoplus_{k=1}^m \Pi_k(b)\otimes 1_{\C^{n_k}},\quad b\in B,\]
where $n_k$ is the multiplicity of $\Pi_i$. 
Note that we have 
\[\Pi(B)''=\bigoplus_{k=1}^m \B(H_k)\otimes 1_{\C^{n_k}}.\]
Let $z_k=1_{H_k}\otimes 1_{\C^{n_k}}$. 
By the Kaplansky density theorem, we have $\|[z_k,T]\|\leq \delta$. 
Let 
\[\Delta(T)=\sum_{k=1}^m z_kTz_k.\]
Then $\|T-\Delta(T)\|\leq m\delta$. 

We take a system of matrix units $\{e^{(k)}_{ij}\}_{1\leq i,j\leq n_k}$ in $\B(\C^{n_k})$ and express $\Delta(T)$ as 
\[\Delta(T)=\bigoplus_{k=1}^m\sum_{1\leq i,j\leq n_k}T^{(k)}_{ij}\otimes e^{(k)}_{ij},\]
where $T^{(k)}_{ij}\in \B(H_k)$. 
For $b\in B_1$, we have  
\[\|[T^{(k)}_{ij},\Pi_k(b)]\|\leq \|[\Delta(T),\Pi(b)]\|\leq (2m+1)\delta.\]
By Lemma \ref{scalar}, there exist $\lambda^{(k)}_{ij}\in \C$ satisfying 
$\|T^{(k)}_{ij}-\lambda^{(k)}_{ij}1_{H_k}\|\leq3(2m+1)\delta$. 
Let 
\[T_1=\bigoplus_{k=1}^m \left(1_{H_k}\otimes \sum_{1\leq i,j\leq n_k}\lambda^{(k)}_{ij}e^{(k)}_{ij}\right).\]
Then $T_1\in \Pi(B)'$, and 
\[\|T_1-T\|\leq\|T_1-\Delta(T)\|+\|\Delta(T)-T\|\leq  3(2m+1)\delta \max_{1\leq k\leq m}n_k^2+m\delta,\]
which shows the statement. 
\end{proof}

\begin{proof}[Proof of Theorem \ref{purely infinite}] 
When $B$ is not unital, it is stable because it is purely infinite and simple. 
In this case, we can choose a projection $p\in B$ satisfying 
\[(A,B,E)\cong (pAp\otimes \K,pBp\otimes \K,E|_{pAp}\otimes \id_\K).\]
Thus we may and do assume that $A$ and $B$ are unital to prove the statement.  

Assume on the contrary that the inclusion $A\supset B$ does not have the property (BEK). 
Then for each natural number $n\in \N$, there exist $x_n,y_n\in A$ with $\|x_n\|=\|y_n\|=1$ satisfying 
\[\sup_{b\in B_1}\|x_nby_n\|\leq \frac{1}{n}.\]
We choose a free ultra-filter $\omega\in \beta\N\setminus\N$. 
Note that $A^\omega$ is purely infinite and simple (see \cite[Proposition 6.2.6]{R02}). 
Then $A^\omega\supset B^\omega$ is an inclusion of purely infinite simple C$^*$-algebras with a conditional 
expectation $E^\omega:A^\omega\to B^\omega$ given by 
\[E^\omega([(x_n)])=[(E(x_n))].\]
Note that we have $\Ind E^\omega=\Ind E$. 
We define $x,y\in A^\omega$ by $x=[(x_n)]$ and  $y=[(y_n)]$. 
Then $\|x\|=\|y\|=1$ and $xB^\omega y=\{0\}$ by construction.    

As was observed in \cite[Section 3]{BEK93}, the C$^*$-algebra generated by $B^\omega x B^\omega$, $B^\omega y B^\omega$, 
and $B^\omega$ is not prime. 
Thus Lemma \ref{intermediate} implies that $A^\omega\supset B^\omega$ is not irreducible. 
Let $r=[(r_n)]\in A^\omega\cap {B^\omega}'$ with $\|r\|=1$. 
Then we have 
\[\lim_{n\to\omega}\sup_{b\in B_1}\|br_n-r_nb\|=0.\]
We choose an irreducible representation $(\Pi,H)$ of $A$. 
Then Lemma \ref{commutant} shows that there exist $Q_n\in \Pi(B)'$ satisfying 
\[\lim_{n\to \omega}\|\Pi(r_n)-Q_n\|=0.\]
We may assume that $\{Q_n\}$ is bounded. 
Since $\Pi(B)'$ is finite dimensional, the norm limit $\lim_{n\to \omega}Q_n$ exists. 
Since $B$ is simple, the restriction of $\Pi$ to $B$ is faithful, and the norm limit $\lim_{n\to\omega}r_n$ exists too. 
This means that $r\in A\cap B'=\C$ and $A^\omega\supset B^\omega$ is irreducible, which is a contradiction. 
\end{proof}

\begin{cor}\label{pipo} Let $A$ be a separable purely infinite simple C$^*$-algebra, and assume that $\rho\in \End(A)_0$ 
is irreducible and $d(\rho)>1$.  
Then $\rho$ is properly outer. 
\end{cor}

Although the following result is not used later, it is of interest in its own right.

\begin{theorem} Let $A\supset B$ be an irreducible inclusion of separable simple C$^*$-algebras with a conditional expectation  
$E:A\to B$ of finite index. 
Let $\K_B(\cE_E)\supset A$ be its basic construction, and let $E_1:\K_B(\cE_E)\to A$ be the dual conditional expectation. 
Let $\iota:B\hookrightarrow A$ be the inclusion map. 
Then the following conditions are equivalent: 
\begin{itemize}
\item[$(1)$] The inclusion $A\supset B$ has the property (BEK). 
\item[$(2)$] The inclusion $\K_B(\cE_E)\supset A$ has the property (BEK). 
\item[$(3)$] For any $x\in A$ and any non-zero hereditary C$^*$-subalgebra $C$ of $B$, 
\[\inf\{\|c(x-E(x))c\|;\; c\in C_+,\; \|c\|=1\}=0.\]
\item[$(4)$] For any $y\in \K_B(\cE_E)$ and any non-zero hereditary C$^*$-subalgebra $D$ of $A$, 
\[\inf\{\|d(y-E_1(y))d\|;\; d\in D_+,\; \|d\|=1\}=0.\]
\end{itemize}
If moreover $B$ belongs to either the class $\cC_1$ or $\cC_2$, the above conditions are further equivalent to 
the following two conditions:
\begin{itemize}
\item[$(5)$] Every irreducible component of $\iota\biota$ not equivalent to $\id_A$ is properly outer.  
\item[$(6)$] Every irreducible component of $\biota\iota$ not equivalent to $\id_B$ is properly outer.  
\end{itemize}
\end{theorem}

\begin{proof} Since the conditions (1)-(4) persist after taking tensor product with $\K$ and passing to the corners by 
a non-zero projection in $M(B)$, we may and do assume that $A,B\in \cC_1$ to prove the theorem. 

(1)$\Rightarrow $(5).  We choose an irreducible representation $\Pi$ of $A$ whose restriction to $B$ is irreducible. 
Then 
\[1=\dim \Hom_B(\Pi\circ\iota,\Pi\circ\iota)=\dim \Hom_A(\Pi,\Pi\circ\iota\biota).\]
Thus if $\rho\in \End(A)_0$ is an irreducible component not equivalent to $\id_A$, we have 
$\dim\Hom_A(\Pi,\Pi\circ\rho)=0$, 
which shows that $\rho$ is properly outer by Lemma \ref{containment}. 

(5)$\Rightarrow$(4). Note that the dual inclusion $\K_B(\cE_E)\supset A$ is isomorphic to $B\supset\biota(A)$. 
Thus the statement follows from the crossed product type decomposition of the inclusion $\K_B(\cE_E)\supset A$ (see Section 2.3). 

(4)$\Rightarrow$(2). We denote $\lambda=(\Ind E)^{-1}$. Let $x_1,x_2\in \K_B(\cE_E)$ with $\|x_1\|=\|x_2\|=1$. 
Our task is to show $\sup_{a\in A_1}\|x_1ax_2\|\geq \lambda^2$. 
Since $\|x_1ax_2\|=\||x_1|a|x_2|\|$, we may and do assume that $x_1,x_2$ are positive. 
By the Pimsner-Popa inequality, we have $\|E_1(x_i)\|\geq \lambda$ for $i=1,2$. 
Let $0<\varepsilon$ be a sufficiently small constant.  
Then by assumption we can choose $a_1,a_2\in A_+$ of norm 1 satisfying 
\[\|a_i(x_i-E_1(x_i))a_i\|\ <\varepsilon,\]
\[\|a_iE_1(x_i)a_i\|> \|E_1(x_i)\|-\varepsilon\geq \lambda-\varepsilon,\]
for $i=1,2$. 
Since $A$ is simple, there exists $a\in A_1$ satisfying 
\[\|a_1E_1(x_1)a_1aa_2E_1(x_2)a_2\|\geq (\lambda -\varepsilon)^2,\]
and so
\begin{align*}
\|x_1a_1aa_2x_2\|&\geq \|a_1x_1a_1aa_2x_2a_2\| \\
 &\geq \|a_1E_1(x_1)a_1aa_2E_1(x_2)a_2\|-2\varepsilon \\
&\geq (\lambda-\varepsilon)^2-2\varepsilon.
\end{align*}
Since $\varepsilon>0$ is arbitrary, we obtain (2).  

The implications (2)$\Rightarrow$(6)$\Rightarrow$(3)$\Rightarrow$(1) follow from the same arguments. 
\end{proof}

\begin{remark}
It was shown in \cite[Theorem 7.5, Corollary 7.6]{I02} that the above condition holds for every finite depth inclusion. 
The condition (3) is adapted as the definition of the outerness of $E$ in \cite{O01}, which is essentially  
the same as the pinching property in \cite[Definition 3.13]{R23}.  
\end{remark}
\section{Proof of Theorem \ref{main1}}
As a corollary of Corollary of \ref{pipo}, we get the following. 

\begin{cor}\label{main1pi} Theorem \ref{main1} is true if $A^\alpha$ is purely infinite. 
\end{cor}
\begin{proof}
We choose a non-zero projection $e\in A^\alpha$ with $[e]_0=0$ in $K_0(A^\alpha)$. 
Then the inclusion $A\supset A^\alpha$ is isomorphic to $eAe\otimes \K\supset eA^\alpha e\otimes \K$ 
if $A^\alpha$ is not unital, and to a corner inclusion of it if $A^\alpha$ is unital. 
Thus $A\supset A^\alpha$ has the property (BEK) if and only if $eAe\supset eA^\alpha e$ has the same property, 
and we may and do assume that $A^\alpha\in \cC_2$ to prove the statement.  

Theorem \ref{pipo} implies that the endomorphisms $\calpha_\pi$, $\pi\in \hG\setminus\{1\}$, $d(\pi)>1$, are  
properly outer. 
Assume that $\pi\in \hG\setminus\{1\}$ has $d(\pi)=1$. 
Note that $\calpha_\pi$ is an automorphism and it is outer thanks to the irreducibility of $A\supset A^\alpha$. 
Thus it is properly outer by Kishimoto's theorem \cite[Lemma 1.1]{K81}. 
\end{proof}

\begin{remark}\label{universality} 
When $A^\alpha$ belongs to $\cC_2$, we can apply Doplicher-Roberts' construction \cite[Theorem 5.1]{DR89I} to 
$\cA=A^\alpha$, $\Delta=\{\calpha_\cH\}_{\cH\in \cU(G,\alpha)}$, and 
$\varepsilon(\calpha_{\cH_1},\calpha_{\cH_2})=\theta(\cH_1,\cH_2)$, and we obtain the universal C$^*$-algebra $\cB$ containing 
$\cA$ with a $G$-action $\talpha$ satisfying the following properties (1)-(4): 
\begin{itemize}  
\item[(1)] $\cA=\cB^{\talpha}$. 
\item[(2)] For each $\cH\in \cU(G,\alpha)$, there exists a copy $\tH\in \cU(G,\talpha)$, as a $G$-space, of $\cH$ such that   
$\cB$ is generated by $\cA$ and $\cup_{\cH\in \cU(G,\alpha)}\tH$.  
\item[(3)] The Hilbert space $\tH$ implements $\calpha_\cH$ in the sense that $Vx=\calpha_\cH(x)V$ holds for all 
$V\in \tH$ and $x\in \cA$. 
\item[(4)] $\varepsilon(\calpha_{\cH_1},\calpha_{\cH_2})=\theta(\tH_1,\tH_2)$.  
\end{itemize}
Moreover the inclusion $\cB\supset \cA$ is irreducible. 

By the universality of $\cB$, there exists a $G$-equivariant surjection from $\cB$ onto $A$ that is the identity on $A^\alpha$. 
On the other hand, we can apply Corollary \ref{main1pi} to $\cB$ and see that $\cB$ is simple. 
Therefore the map from $\cB$ to $A$ is an isomorphism. 
This means that $A$ has universality with respect to (1)-(4), which we will use in the proof of Theorem \ref{main1} in the general case.  
\end{remark}

As a consequence of Corollary \ref{main1pi}, we get the following lemma. 
Recall that if $\beta$ is an ergodic action of $G$ on a C$^*$-algebra $B$, i.e. $B^\beta=\C$, the C$^*$-algebra $B$ 
is necessarily nuclear (see \cite[Lemma 22]{WII}, \cite[Proposition 3]{DLRZ02}).

\begin{lemma}\label{ergodic} 
Let the notation be as in the assumption of Theorem \ref{main1}.   
Let $\beta$ be an ergodic action of $G$ on a C$^*$-algebra $B$, and let $\gamma$ be the diagonal action 
$\gamma_g=\alpha_g\otimes \beta_g$ of $G$ on $A\otimes B$. 
Then the fixed point algebra $(A\otimes B)^\gamma$ is simple. 
More over there exists a conditional expectation $F:(A\otimes B)^\gamma\to A^\alpha\otimes \C$. 
\end{lemma}

\begin{proof} Note that the restriction $F$ of $E_\alpha\otimes E_\beta$ to $(A\otimes B)^\gamma$ is 
a faithful conditional expectation from $(A\otimes B)^\gamma$ onto $A^\alpha\otimes \C$. 

Replacing $A$ and $\alpha$ with $A\otimes \cO_\infty$ and $\alpha\otimes \id_{\cO_\infty}$ respectively, 
we may assume that $A^\alpha$ is purely infinite. 
Moreover, passing to the corner by a projection in $A^\alpha$, we may and do assume that $A^\alpha\in \cC_2$.  

We first recall basic facts about ergodic actions (see \cite[Proposition 2.1,Theorem 4.1]{HLS}, \cite[Theorem 1]{WasI}). 
The conditional expectation $E_\beta:B\to B^\beta=\C$ is a trace, and we denote it by $\tau$ for simplicity. 
There exist non-negative integers $m_\pi\leq d(\pi)$ with $\dim B^\beta(\pi)=m_\pi d(\pi)$ for $\pi\in \hG$. 
We can choose $X(\pi)_a\in B^\beta_1(\pi)$, $a=1,2,\cdots,m_\pi$ such that $\{X(\pi)_{ai}\}_{a,i}$ form an orthonormal 
basis of $B^\beta(\pi)$ with respect to the inner product given by $\tau$. 
Then we have 
\[\sum_{i=1}^{d(\pi)}X(\pi)_{ai}X(\pi)_{bi}^*=\delta_{a,b}d(\pi).\]

Let 
\[W(\pi)_a=\frac{1}{\sqrt{d(\pi)}}\sum_{i=1}^{d(\pi)}\psi(\pi)_i\otimes X(\pi)_{ai}^*.\]
Then $W(\pi)_a$, $a=1,2,\ldots,m_\pi$, are isometries in $(A\otimes B)^\gamma$ with mutually orthogonal ranges.  
Since the linear span of $\cup_{\pi_1,\pi_2}A^\alpha(\pi_1)\otimes B^\beta(\pi_2)$ is dense in $A\otimes B$, 
the linear span of $\cup_{\pi_1,\pi_2}E_\gamma(A^\alpha(\pi_1)\otimes B^\beta(\pi_2))$ is dense in $(A\otimes B)^\gamma$. 
On the other hand, we have $A^\alpha(\pi)=A^\alpha\cH_\pi$, and $E_\gamma(A^\alpha(\pi_1)\otimes B^\beta(\pi_2))$ survives 
only when $\pi_2$ is equivalent to $\overline{\pi_1}$. 
Thus we can see that the linear span of $\cup_{\pi,a}(A^\alpha\otimes 1)W(\pi)_a$ is dense in $(A\otimes B)^\gamma$. 
Note that 
\[W(\pi)_a(x\otimes 1)=(\calpha_\pi(x)\otimes 1) W(\pi)_a,\]
holds for all $x\in A^\alpha$. 
Since $F$ is faithful and $\calpha_\pi$, $\pi\in \hG\setminus\{1\}$, are all properly outer, we can conclude that $(A\otimes B)^\gamma$ 
is simple as in the case of a crossed product by a discrete group.   
\end{proof}

We proceed to the proof of Theorem 1.1 in the general case, which requires several reduction steps. 
Since the inclusion $A\supset A^\alpha$ has the property (BEK) if and only if $A\otimes \K\supset A^\alpha\otimes \K$ 
has the same property, we may assume $A^\alpha\in \cC_1$ to prove Theorem \ref{main1}. 
Our task is to show that $\calpha_\pi$ is properly outer for $\pi\in \hG$ with $d(\pi)>1$. 

We fix $\pi_0\in \hG$ with $d(\pi_0)>1$ and assume on the contrary that $\calpha_{\pi_0}$ is not properly outer. 
Then the proof of \cite[Theorem 7.5]{I02}, which is an adaptation of Kishimoto's argument in \cite{K81} to the endomorphism case, 
shows that there would exist $b\in M(A^\alpha)$, a non-zero hereditary 
C$^*$-subalgebra $C$ of $A^\alpha$, 
and $\delta>0$ such that for any irreducible representation $(\Pi,H)$ of $A^\alpha$, there exists a finite rank self-adjoint operator 
$T$ on $H$ and isometry $V\in \Hom_A(\Pi,\Pi\circ \calpha_{\pi_0})$ satisfying 
\[\Pi(c)^*(\Pi(b)V+V^*\Pi(b)^*+T)\Pi(c)\geq \delta \Pi(c^*c),\]
for all $c\in C$.  
Let $\cL_\pi=\Hom_{A^\alpha}(\Pi,\Pi\circ \calpha_\pi)$, which is a finite dimensional Hilbert space in $\B(H)$. 
Let $D$ be the C$^*$-algebra generated by $\cup_{\pi\in \hG}\Pi(A^\alpha)\cL_\pi$. 
As was observed in the proof of \cite[Theorem 7.5]{I02}, if $D$ is simple and there exists a conditional expectation $F$ from 
$D$ onto $\Pi(A^\alpha)$, we can get a contradiction. 
Therefore Theorem \ref{main1} is reduced to the following lemma: 

\begin{lemma}\label{sim-con1} Under the assumption in Theorem \ref{main1}, we assume $A^\alpha\in \cC_1$. 
Let $(\Pi,H)$ be an irreducible representation of $A^\alpha$, let $\cL_\pi=\Hom_{A^\alpha}(\Pi,\Pi\circ \calpha_\pi)$ for 
$\pi\in \hG$, and let $D$ be the C$^*$-algebra generated by $\cup_{\pi\in \hG}\Pi(A^\alpha)\cL_\pi$. 
Then $D$ is simple and there exists a conditional expectation from $D$ onto $\Pi(A^\alpha)$.  
\end{lemma}

For simplicity, we suppress the symbol $\Pi$ in $\Pi(A^\alpha)$, and treat $A^\alpha$ as a subalgebra of $\B(H)$ irreducibly 
acting on $H$. 
For each $\pi\in \hG$, we choose an orthonormal basis $\{V(\pi)_p\}_{p=1}^{m_\pi}$ of $\cL_\pi$.  
Note that we have $m_1=1$, and we can choose $V(1)_1=1$. 
Then as in \cite[Lemma 7.1]{I02}, we can show that there exist 
$C_{(\pi,p),(\sigma,q)}^{(\mu,r)}\in (\calpha_\mu,\calpha_\pi\calpha_\sigma)$ and 
$c_{(\pi,p),(\overline{\pi},q)}\in \C$ satisfying 
\begin{equation}\label{C1}
V(\pi)_pV(\sigma)_q=\sum_{\mu,r}C_{(\pi,p),(\sigma,q)}^{(\mu,r)}V(\mu)_r,
\end{equation}
\begin{equation}\label{C2}
V(\pi)_p^*=\sum_{q}c_{(\pi,p),(\bpi,q)}R_\pi^*V(\bpi)_q. 
\end{equation}
Let $D_0$ be the linear span of $\cup_{\pi\in \hG}A^\alpha\cL_\pi$. 
Then the above relations show that $D_0$ is a dense $*$-subalgebra of $D$, and every approximate unit of $A^\alpha$ is an approximate unit of $D$. 
In particular, we have $\overline{A^\alpha D}^{\|\cdot \|}=D$. 

We prove Lemma \ref{sim-con1} by using Doplicher-Roberts' endomorphism crossed product in \cite{DR89I}. 
As it is formulated only for unital C$^*$-algebras, we show that Lemma \ref{sim-con1} is reduced to the same statement 
with $A^\alpha\in \cC_2$. 
 
We first show that Lemma \ref{sim-con1} is reduced to the case where $A^\alpha$ is purely infinite. 
To see it, we fix an irreducible representation $(\Pi_1,H_1)$ of $\cO_\infty$, and replace $A$ and $(\Pi,H)$ 
with $A\otimes \cO_\infty$ and $(\Pi\otimes \Pi_1,H\otimes H_1)$ respectively. 
If Lemma \ref{sim-con1} holds for purely infinite $A^\alpha\in \cC_1$, we see that $D\otimes \cO_\infty$ is simple and there exists 
a conditional expectation from $D\otimes \cO_\infty$ onto $A^\alpha\otimes \cO_\infty$.  
This implies that $D$ is simple and there exists a conditional expectation from $D$ onto $A^\alpha$. 

Next we show that Lemma \ref{sim-con1} for purely infinite $A^\alpha\in \cC_1$ is reduced to the same statement for 
$A^\alpha\in \cC_2$. 
We choose a system of matrix units $\{e_{i,j}\}_{i,j\geq 1}$ in $A^\alpha$ such that $[e_{11}]_0=0$ in $K_0(A^\alpha)$ and 
$\{\sum_{i=1}^n e_{ii}\}_{n=1}^\infty$ is an approximate unit of $A^\alpha$, and hence converges to 1 in the strict topology.  
We denote $e=e_{11}$ and $p_n=\sum_{i=1}^n e_{ii}$ for simplicity. 
Note that $e$ is a full projection of $D$ because $\{p_n\}_{n=1}^\infty$ is an approximate unit of $D$ too. 
Note that $\{\calpha_\pi(p_n)\}_{n=1}^\infty$ is an approximate unit of $A^\alpha$ as $\calpha_\pi\in \End(A^\alpha)_0$. 

For each $\pi\in \hG$, we choose a partial isometry $v_\pi\in A^\alpha$ satisfying $v_\pi v_\pi^*=e$ and 
$v_\pi^*v_\pi=\calpha_\pi(e)$. 
Then 
\[U_\pi=\sum_{i=1}^\infty e_{i1}v_\pi\calpha_\pi(e_{1i})\]
converges to a unitary in $M(A^\alpha)$ in the strict topology, and $\Ad U_\pi\cdot \calpha_\pi(e_{ij})=e_{ij}$. 
Replacing $\cH_\pi$ with $U_\pi\cH_\pi$ implies replacing $\cL_\pi$ with $U_\pi\cL_\pi$, which does not change $D$. 
Thus we may do assume $\calpha_\pi(e_{ij})=e_{ij}$, and
\[(D\supset A^\alpha,\calpha_\pi)\cong 
(eDe\otimes \K\supset eA^\alpha e\otimes \K,\calpha_\pi|_{eA^\alpha e}\otimes \id_\K).\]
Therefore if the same statement as in Lemma \ref{sim-con1} for $A^\alpha\in \cC_2$ holds, we can see that 
$eDe$ is simple and there exists a conditional expectation from $eDe$ onto $eA^\alpha e$, 
which shows that Lemma \ref{sim-con1} holds. 

Now Theorem \ref{main1} is reduced to the following lemma: 
 
\begin{lemma}\label{sim-con2} Under the assumption in \ref{main1}, we assume $A^\alpha\in \cC_2$. 
Let $(\Pi,H)$ be an irreducible representation of $A^\alpha$, let $\cL_\pi=\Hom_{A^\alpha}(\Pi,\Pi\circ \calpha_\pi)$ for 
$\pi\in \hG$, and let $D$ be the C$^*$-algebra generated by $\Pi(A^\alpha)$ and $\cup_{\pi\in \hG}\cL_\pi$. 
Then $D$ is simple and there exists a conditional expectation from $D$ onto $\Pi(A^\alpha)$.  
\end{lemma}

\begin{remark}\label{MC} 
Let $\Rep(A^\alpha)_0$ be the category representations of $A^\alpha$ unitarily equivalent to 
finite direct sums of irreducible representations of $A^\alpha$. 
Then Lemma \ref{finite} shows that $\Rep(A^\alpha)_0$ is a module category of $\cU(G,\alpha)$ via 
$\Pi\otimes \cH:=\Pi\circ \calpha_\cH$. 
It is known that there is a correspondence between module categories of $\cU(G)$ and ergodic $G$-actions 
(see \cite[Theorem 6.4]{CY13}). 
In fact, we prove Lemma \ref{sim-con2} by describing the structure of $D$ in terms of an ergodic $G$-action. 
\end{remark}

From now on, we assume $A^\alpha\in \cC_2$. 
We choose an orthonormal basis $\{V(\pi)_p\}_p$ of $\cL_\pi$ for each $\pi\in \hG$. 
Then there exists $C_{(\pi,p),(\sigma,q)}^{(\mu,r)}\in (\calpha_\mu,\calpha_\pi\calpha_\sigma)$ 
and $c_{(\pi,p),(\overline{\pi},q)}\in \C$ satisfying Eq.(\ref{C1}) and (\ref{C2}).  

Instead of working on $D$ directly, we first consider its universal counterpart. 
Let $\cD$ be the universal C$^*$-algebra generated by $A^\alpha$ and isometries $V'(\pi)_p$, 
$\pi\in \hG$, $p=1,2,\cdots,m_\pi$, satisfying the following relations:
\begin{equation}\label{R1}
V'(\pi)_p^*V'(\pi)_q=\delta_{p,q},
\end{equation}
\begin{equation}\label{R2}
V'(\pi)_px=\calpha_\pi(x)V'(\pi)_p,\quad x\in A^\alpha, 
\end{equation}
\begin{equation}\label{R3}
V'(\pi)_pV'(\sigma)_q=\sum_{\mu,r}C_{(\pi,p),(\sigma,q)}^{(\mu,r)}V'(\mu)_r,
\end{equation}
\begin{equation}\label{R4}
V'(\pi)_p^*=\sum_{q}c_{(\pi,p),(\bpi,q)}R_\pi^*V'(\bpi)_q. 
\end{equation}

\begin{lemma} Let the notation be as above. 
\begin{itemize}
\item[$(1)$] For each $\cH\in \cU(G,\alpha)$, there exists $\calpha'_\cH\in \End(\cD)$ satisfying $\calpha'_\cH|_{A^\alpha}=\calpha_\cH$, 
and $\calpha'_\cH(V'(\pi)_p)=\theta(\pi,\cH)V'(\pi)_p$. 
\item[$(2)$] We set $\calpha'_T=T$ for $T\in (\cH_1,\cH_2)_G$. 
Then $\calpha'$ is a tensor functor from $\cU(G,\alpha)$ to $\End(\cD)$. 
\end{itemize}
\end{lemma}

\begin{proof}(1) By the universality of $\cD$, it suffices to show that $\calpha'_\cH$ preserves the relations 
(\ref{R1})-(\ref{R4}). 
Since $\theta(\pi,\cH)$ is unitary, the relation (\ref{R1}) is preserved. 
We can see that (\ref{R2}) is preserved from  
\[\theta(\pi,\cH)V'(\pi)_p\calpha_\pi(x)=\theta(\pi,\cH)\calpha_\pi\calpha_\cH(x)V'(\pi)_p
=\calpha_\cH\calpha_\pi(x)\theta(\pi,\cH)V'(\pi)_p.\]
To see that (\ref{R3}) is preserved, it suffices to show 
\[\theta(\pi,\cH)V'(\pi)_p\theta(\sigma,\cH)V'(\sigma)_q=\sum_{\mu,r}\calpha_\cH(C_{(\pi,p),(\sigma,q)}^{(\mu,r)})
\theta(\mu,\cH)V'(\mu)_r.\]
The left-hand side is equal to 
\begin{align*}
\lefteqn{\theta(\pi,\cH)\calpha_\pi(\theta(\sigma,\cH))V'(\pi)_pV'(\sigma)_q
=\theta(\pi\otimes \sigma,\cH)\sum_{\mu,r}C_{(\pi,p),(\sigma,q)}^{(\mu,r)}V'(\mu)_r } \\
 &=\sum_{\mu,r}\left(\theta(\pi\otimes \sigma,\cH)C_{(\pi,p),(\sigma,q)}^{(\mu,r)}\theta(\cH,\mu)\right)\theta(\mu,\cH)V'(\mu)_r,\\
\end{align*}
which coincides with the right-hand side thanks to Eq.(\ref{P3}). 
Finally to see that the relation (\ref{R4}) is preserved, it suffices to show 
\[V'(\pi)_p^*\theta(\cH,\pi)=\sum_{q}c_{(\pi,p),(\bpi,q)}\calpha_\cH(R_\pi^*)\theta(\bpi,\cH)V'(\bpi)_q.\]
The left-hand side is 
\begin{align*}
\lefteqn{\sum_{q}c_{(\pi,p),(\bpi,q)}R_\pi^*V'(\bpi)_q\theta(\cH,\pi)
 =\sum_{q}c_{(\pi,p),(\bpi,q)}R_\pi^*\calpha_{\bpi}(\theta(\cH,\pi))V'(\bpi)_q} \\
 &=\sum_{q}c_{(\pi,p),(\bpi,q)}\left(R_\pi^*\calpha_{\bpi}(\theta(\cH,\pi))\theta(\cH,\bpi)\right)\theta(\bpi,\cH)V'(\bpi)_q. 
\end{align*}
Since Eq.(\ref{P2}) and(\ref{P3}) imply  
\[\theta(\bpi,\cH)\calpha_{\bpi}(\theta(\pi,\cH))R_\pi=\theta(\bpi\otimes\pi,\cH)R_\pi=\calpha_\cH(R_\pi),\]
the relation (\ref{R4}) is preserved, and $\calpha'_\cH$ is well-defined. 

(2) It suffices to show that $T\in (\cH_1,\cH_2)_G$ satisfies 
\[T\theta(\pi,\cH_1)V'(\pi)_p=\theta(\pi,\cH_2)V'(\pi)_pT.\] 
The right-hand side is 
\[\theta(\pi,\cH_2)\calpha_\pi(T)V'(\pi)_p=T\theta(\pi,\cH_1)V'(\pi)_p,\]
by Eq.(\ref{P3}), and the equality holds. 
\end{proof}

\begin{proof}[Proof of Lemma \ref{sim-con2}] 
We apply \cite[Theorem 5.1]{DR89I} to $\cA=\cD$, $\Delta=\{\calpha'_\cH\}_{\cH\in \cU(G,\alpha)}$, 
and $\varepsilon(\calpha'_{\cH_1},\calpha'_{cH_2})=\theta(\cH_1,\cH_2)$. 
Although \cite[Theorem 5.1]{DR89I} is stated under the condition that $\cA$ has trivial center, their construction 
itself works without this condition. 
Thanks to Remark \ref{universality}, the resulting algebra $\cB$ satisfies the following properties:  
\begin{itemize}
\item[(1)] The C$^*$-algebra $\cB$ is generated by $\cD$ and $A$ with a common C$^*$-subalgebra $A^\alpha$. 
\item[(2)] For every $\cH\in \cU(G,\alpha)$, the Hilbert space $\cH$ implements $\calpha'_\cH$.   
\item[(3)] There exists a $G$-action $\alpha'$ on $\cB$ extending $\alpha$ such that $\cB^{\alpha'}=\cD$. 
\end{itemize}

For $\pi \in \hG$, let $Y(\pi)_{pi}=V'(\pi)_p^*\psi(\pi)_i$. 
Then $Y(\pi)_a\in \cB^{\alpha'}(\pi)$. 
Since $\cD$ is generated by $A^\alpha$ and $\cup_{\pi\in \hG}\cL_\pi$, and $A$ is generated by $A^\alpha$ and 
$\cup_{\pi\in \hG}\cH_\pi$, the C$^*$-algebra $\cB$ is generated by $A$ and $\{Y(\pi)_{pi}\}_{p,i}$.  
We claim that $Y(\pi)_{pi}$ commutes with $A$. 
It commutes with $\cA^\alpha$ by construction. 
Since the linear span of $\cup_{\sigma\in \hG}A^\alpha\cH_\sigma$ is dense in $A$, it suffices to show that $Y(\pi)_{pi}$ 
commutes with $\psi(\sigma)_j$. 
Indeed, 
\begin{align*}
\lefteqn{\psi(\sigma)_jY(\pi)_{pi}=\calpha'_\sigma(V'(\pi)_p^*)\psi(\sigma)_j\psi(\pi)_i} \\
 &=V'(\pi)_p^*\theta(\sigma,\pi)\psi(\sigma)_j\psi(\pi)_i =V'(\pi)_p^*\psi(\pi)_i\psi(\sigma)_j=Y(\pi)_{pi}\psi(\sigma)_j, 
\end{align*}
which shows the claim. 

We next claim that the linear span $B_0$ of $\{Y(\pi)_{ai}\}_{\pi,a,i}$.is a $*$-algebra. 
Indeed, 
\begin{align*}
Y(\pi)_{pi}^*&=\psi(\pi)_i^*V'(\pi)_p=\sqrt{d(\pi)}R_\pi^*\psi(\bpi)_iV'(\pi)_p=\sqrt{d(\pi)}R_\pi^*\calpha'_{\bpi}(V'(\pi)_p)\psi(\bpi)_i  \\
 &=\sqrt{d(\pi)}R_\pi^*\theta(\pi,\bpi)V'(\pi)_p\psi(\bpi)_i =\sqrt{d(\pi)}\bR_\pi^*V'(\pi)_p\psi(\bpi)_i\\
 &=\sqrt{d(\pi)}\bR_\pi^*\sum_{q} \overline{c((\pi,p),(\bpi,q))}V'(\bpi)_q^*R_\pi\psi(\bpi)_i\\
  &=\sqrt{d(\pi)}\sum_{q} \overline{c((\pi,p),(\bpi,q))}V'(\bpi)_q^*\calpha_{\bpi}(\bR_\pi^*)R_\pi\psi(\bpi)_i\\
  &=\frac{1}{\sqrt{d(\pi)}}\sum_{q} \overline{c((\pi,p),(\bpi,q))}Y(\bpi)_{qi},\\
\end{align*}
shows that $B_0$ is closed under $*$. 
For the product, we have
\begin{align*}
\lefteqn{Y(\sigma)_{qj}Y(\pi)_{pi}=V'(\sigma)_q^*\psi(\sigma)_jV'(\pi)_p^*\psi(\pi)_i
 =V'(\sigma)_q^*\calpha'_\sigma(V'(\pi)_p^*)\psi(\sigma)_j\psi(\pi)_i}\\ 
 &=V'(\sigma)_q^*V'(\pi)_p^*\theta(\sigma,\pi)\psi(\sigma)_j\psi(\pi)_i
=\sum_{\mu,r}V'(\mu)_r^*{C_{(\pi,p),(\sigma,q)}^{(\mu,r)}}^*\psi(\pi)_i\psi(\sigma)_j\\
&=\sum_{\mu,r}\sum_{\nu,k}d(\nu)V'(\mu)_r^*{C_{(\pi,p),(\sigma,q)}^{(\mu,r)}}^*
E_\alpha(\psi(\pi)_i\psi(\sigma)_j\psi(\nu)_k^*)\psi(\nu)_k. 
\end{align*}
Since $C_{(\pi,p),(\sigma,q)}^{(\mu,r)}\in (\calpha_\mu,\calpha_{\pi}\calpha_\sigma)$ and 
$E_\alpha(\psi(\pi)_i\psi(\sigma)_j\psi(\nu)_k^*)\in (\calpha_\nu,\calpha_\pi\calpha_\sigma)$, 
the product ${C_{(\pi,p),(\sigma,q)}^{(\mu,r)}}^*E_\alpha(\psi(\pi)_i\psi(\sigma)_j\psi(\nu)_k^*)$ survives only 
if $\mu=\nu$, and 
\begin{align*}
Y(\sigma)_{qj}Y(\pi)_{pi}
&=\sum_{\mu,r,s}d(\mu)\inpr{E_\alpha(\psi(\pi)_i\psi(\sigma)_j\psi(\mu)_k^*)}{C_{(\pi,p),(\sigma,q)}^{(\mu,r)}}V'(\mu)_r^*\psi(\mu)_k \\
 &=\sum_{\mu,r,s}d(\mu)\inpr{E_\alpha(\psi(\pi)_i\psi(\sigma)_j\psi(\mu)_k^*)}{C_{(\pi,p),(\sigma,q)}^{(\mu,r)}}Y(\mu)_{rk}. 
\end{align*}
Therefore the claim is shown. 

Let $B$ be the closure of $B_0$ and let $\beta$ be the restriction of $\alpha'$ to $B$. 
Then $B^\beta(\pi)$ is the linear span of $\{Y(\pi)_{ai}\}_{a,i}$, and in particular we have $B^\beta=\C$, that is, 
the action $\beta$ is ergodic. 
Thus $B$ is nuclear. 
Since $A$ is simple and $\cB$ is generated by two mutually commuting C$^*$-subalgebras $A$ and $B$, we can identify $\cB$ with 
$A\otimes B$. 
Thus Lemma \ref{ergodic} implies that $\cD=\cB^{\alpha'}$ is simple and there exists a conditional expectation from $\cD$ onto $A^\alpha$. 
By the universality of $\cD$, there exists an isomorphism from $\cD$ onto $D$ that is the identity on $A^\alpha$. 
Thus Lemma \ref{sim-con2} is shown. 
\end{proof}

\section{Fej\'er type approximation for compact groups}
In this section, we assume that $G$ is a second countable compact group, and we use the notation in Section 2.4. 
We always use the normalized Haar measure of $G$. 
We equip $C(G)$ with a $G$-action by the left translation $L_g(f)(h)=f(g^{-1}h)$ 
for $f\in C(G)$. 
We denote by $\lambda$ the left regular representation of $G$. 
For $f\in C(G)$, we denote by $M_f\in \B(L^2(G))$ the multiplication operator by $f$. 

We first establish a $G$-equivariant CPAP of $C(G)$, 
which may be interpreted as a Fej\'er type approximation. 
The orthogonality relation shows that we have a completely 
orthonormal system $\{\sqrt{d(\pi)}\pi_{ij}\}_{\pi\in \hG,\; 1\leq i,j\leq d(\pi)}$ of $L^2(G)$. 
We denote by 
\[\{E_{\sqrt{d(\pi)}\pi_{ij},\sqrt{d(\sigma)}\sigma_{kl}}\}_{\pi,\sigma\in \hG,\;1\leq i,j\leq d(\pi),\; 1\leq k,l\leq d(\sigma)}\]
the corresponding system of matrix units in $\B(L^2(G))$. 
For a finite subset $F\subset \hG$, we define a finite rank  projection $P_F\in \B(L^2(G))$ commuting with $\lambda_g$ by 
\[P_F=\sum_{\pi\in F}\sum_{1\leq i,j\leq d(\pi)}E_{\sqrt{d(\pi)}\pi_{ij},\sqrt{d(\pi)}\pi_{ij}}.\]

Mimicking the construction in the proof of \cite[Theorem 2.6.8]{BO08} for the reduced group C$^*$-algebra of a discrete group, 
we define ucp maps $\varphi_F:C(G)\to \B(P_FL^2(G))$ and $\varphi'_F:\B(P_FL^2(G))\to C(G)$ by 
$\varphi_F(f)=P_FM_fP_F$, and 
\[\varphi'_F(E_{\sqrt{d(\pi)}\pi_{ij},\sqrt{d(\sigma)}\sigma_{kl}})=\frac{1}{w(F)}\sqrt{d(\pi)d(\sigma)}\pi_{ij}\overline{\sigma_{kl}},\]
\[w(F)=\sum_{\pi \in F}d(\pi)^2.\]
Then $\varphi_F$ and $\varphi'_F$ are $G$-equivariant in the sense that 
$\varphi_F\circ L_g=\Ad(P_F\lambda_g)\circ \varphi_F$ and $\varphi'_F\circ \Ad (P_F\lambda_g) =L_g\circ \varphi'_F$ hold. 

We define a kernel function $K_F\in C(G)$ associated with $F$ by 
\[K_F(g)=\frac{1}{w(F)}\left| \sum_{\pi\in F}d(\pi)\chi_\pi(g)\right|^2.\]
Recall that the convolution $f_1*f_2$ of $f_1,f_2\in C(G)$ is defined by 
\[f_1*f_2(g)=\int_G f_1(h)f_2(h^{-1}g)dh=\int_Gf_1(gh^{-1})f_2(h)dh.\] 

\begin{lemma} For $f\in C(G)$, we have 
\[\varphi'_F\circ \varphi_F(f)=f*K_F.\]
\end{lemma}

\begin{proof} We have 
\begin{align*}
\lefteqn{\varphi_F(f)=\sum_{\pi,\sigma\in F}\sum_{i,j,k,l}\inpr{M_f\sqrt{d(\sigma)}\sigma_{kj}}{\sqrt{d(\pi)}\pi_{ij}}
E_{\sqrt{d(\pi)}\pi_{ij},\sqrt{d(\sigma)}\sigma_{kl}}} \\
 &=\sum_{\pi,\sigma\in F}\sum_{i,j,k,l}\sqrt{d(\pi)d(\sigma)}
 \int_G f(h)\sigma_{kl}(h)\overline{\pi_{ij}(h)}dh
E_{\sqrt{d(\pi)}\pi_{ij},\sqrt{d(\sigma)}\sigma_{kl}},\\
\end{align*}
\begin{align*}
\lefteqn{\varphi'_F\circ \varphi_F(f)(g)
 = \frac{1}{w(F)}\sum_{\pi,\sigma\in F}\sum_{i,j,k,l}d(\pi)d(\sigma)
 \int_G f(h)\sigma_{kl}(h)\overline{\sigma_{kl}(g)}\overline{\pi_{ij}(h)}\pi_{ij}(g)dh}\\
 &=\frac{1}{w(F)}\sum_{\pi,\sigma\in F}d(\pi)d(\sigma)
 \int_G f(h)\overline{\chi_\sigma(h^{-1}g)}\chi_\pi(h^{-1}g)dh=\int_G f(h)K_F(h^{-1}g)dh, 
\end{align*}
which shows the statement. 
\end{proof}

Popa \cite{P94} introduced the notion of amenability for rigid C$^*$-tensor categories in the context of subfactors. 
As in the case of discrete groups, he formulated a F\o lner sequence and showed that amenability is equivalent to 
the existence of a F\o lner sequence. 
It was observed in \cite{HI98} that the amenability is a property depending only on the fusion algebra structure of the category. 
The following version of the definition of a F\o lner sequence was formulated in {\cite[Definition 4.5]{HI98}}, 
and it is equivalent to Popa's original definition.

\begin{definition} A sequence $\{F_n\}_{n=1}^\infty $ of finite subsets of $\hG$ is said to be a F\o lner sequence if for 
every $\pi\in \hG$, the following holds: 
\[\lim_{n\to\infty}\frac{1}{w(F_n)}\left(\sum_{\sigma\in F_n}\sum_{\mu\in F_n^c}\frac{d(\sigma)d(\mu)}{d(\pi)}N_{\pi,\mu}^\sigma
+\sum_{\sigma\in F_n^c}\sum_{\mu\in F_n}\frac{d(\sigma)d(\mu)}{d(\pi)}N_{\pi,\mu}^\sigma\right)=0,\]
where $F_n^c$ is the complement of $F_n$. 
\end{definition}

It is well-known that the representation category of $G$ is amenable, and $\hG$ always has a F\o lner sequence 
(see, for example, \cite[Theorem 4.5]{Ru96}, \cite[subsection 2.3]{MT07}). 
One easy way to see it is as follows. 
Since every finitely generated subcategory of $\cU(G)$ is equivalent to the representation category of a closed subgroup 
of $SU(n)$, 
it has polynomial growth, and hence has a F\o lner sequence (see \cite[Example 7.6]{HI98}). 
Therefore the whole category has a F\o lner sequence too. 

\begin{example} When $G=\T$, we have $\widehat{\T}=\Z$, and  
$F_n=\{0,1,\ldots,n\}$, $n\in \N$, give a F\o lner sequence. 
The function $K_{F_n}$ in this case is 
\[K_{F_n}(e^{it})=\frac{1}{n+1}\left|\sum_{k=0}^ne^{ikt}\right|^2=\frac{1}{n+1}\frac{\sin^2\frac{(n+1)t}{2}}{\sin^2\frac{t}{2}},\]
which is nothing but the Fej\'er kernel.
\end{example}

\begin{theorem}\label{Fejer} Let $\{F_n\}_{n=1}^\infty$ be a F\o lner sequence for $\hG$. 
Then for all $f\in C(G)$, 
\[\lim_{n\to \infty}\|\varphi'_{F_n}\circ \varphi_{F_n}(f)-f\|_\infty=0.\]
\end{theorem}

\begin{proof} Since the linear span of the matrix coefficients of the irreducible representations is uniformly dense in $C(G)$, 
it suffice to show the statement for $f=\pi_{ij}$ with $\pi\in \hG$. 
For a finite set $F\subset \hG$, we have  
\begin{align*}
\lefteqn{\pi_{ij}*K_{F}(g)
=\frac{1}{w(F)}\sum_{k=1}^{d(\sigma)}\pi_{ik}(g)\sum_{\sigma,\mu\in F}d(\sigma)d(\mu)
 \int_G \overline{\pi_{jk}(h)}\chi_\sigma(h)\overline{\chi_\mu(h)}dh}\\
 &=\frac{1}{w(F)}\sum_{k=1}^{d(\sigma)}\pi_{ik}(g)\sum_{\sigma,\mu\in F}d(\sigma)d(\mu)\sum_{\nu\in \hG}N_{\sigma,\overline{\mu}}^\nu
 \int_G \overline{\pi_{jk}(h)}\chi_\nu(h)dh \\
&=\frac{\pi_{ij}(g)}{w(F)}\sum_{\sigma,\mu\in F}\frac{d(\sigma)d(\mu)}{d(\pi)}N_{\sigma,\overline{\mu}}^\pi.\\
\end{align*}
Since $N_{\sigma,\overline{\mu}}^\pi=N_{\pi,\mu}^\sigma$, and $\sum_{\sigma\in \hG}d(\sigma)N_{\pi,\mu}^\sigma=d(\pi)d(\mu)$, 
we get 
\begin{align*}
\lefteqn{\frac{1}{w(F)}\sum_{\sigma,\mu\in F}\frac{d(\sigma)d(\mu)}{d(\pi)}N_{\sigma,\overline{\mu}}^\pi
=\frac{1}{w(F)}\sum_{\sigma,\mu\in F}\frac{d(\sigma)d(\mu)}{d(\pi)}N_{\pi,\mu}^\sigma} \\
 &=\frac{1}{w(F)}\sum_{\mu\in F}\left(\sum_{\sigma\in \hG}\frac{d(\sigma)d(\mu)}{d(\pi)}N_{\pi,\mu}^\sigma-
 \sum_{\sigma\in F^c}\frac{d(\sigma)d(\mu)}{d(\pi)}N_{\pi,\mu}^\sigma
 \right) \\
 &=1-\frac{1}{w(F)}\sum_{\mu\in F}\sum_{\sigma\in F^c}\frac{d(\sigma)d(\mu)}{d(\pi)}N_{\pi,\mu}^\sigma. \\
\end{align*}
This shows 
\[\|\pi_{ij}-\varphi'_{F_n}\circ \varphi_{F_n}(\pi_{ij})\|_\infty\leq 
\frac{\|\pi_{ij}\|_\infty}{w(F_n)}\sum_{\mu\in F_n}\sum_{\sigma\in F_n^c}\frac{d(\sigma)d(\mu)}{d(\pi)}N_{\pi,\mu}^\sigma\to 0,\quad (n\to\infty).\]
\end{proof}
 
As an application, we show that every nuclear C$^*$-algebra with a compact group action has 
the following equivariant CPAP.  

\begin{lemma}\label{eqCPAP} 
Let $B$ be a unital nuclear C$^*$-algebra with a $G$-action $\beta$. 
Then for any finite subset $F\subset B$ and $\varepsilon>0$, there exists $n\in \N$, ucp maps $\varphi:A\to \M_n(\C)$, 
$\varphi':\M_n(\C)\to B$, and a unitary representation $u$ of $G$ in $\M_n(\C)$ such that $\varphi\circ \beta_g=\Ad u(g)\circ \varphi$, 
$\varphi'\circ \Ad u(g)=\beta_g\circ \varphi'$, and $\|\varphi'\circ \varphi(x)-x\|<\varepsilon$ for all $x\in F$. 
\end{lemma}

\begin{proof} Since $B$ is nuclear and $\beta_G(F)$ is a compact set, there exists $n_0$ and ucp maps $\varphi_0:B\to \M_{n_0}(\C)$ 
and $\varphi'_0:\M_{n_0}(\C)\to B$ satisfying 
\[\|\varphi'_0\circ \varphi_0(\beta_g(x))-\beta_g(x)\|< \varepsilon/2,\quad \forall g\in G,\;\forall x\in F.\] 
We define ucp maps $\Phi:B\to C(G)\otimes \M_{n_0}(\C)$ and $\Phi':C(G)\otimes \M_{n_0}(\C)\to B$ by 
$\Phi(x)(g)=\varphi_0(\beta_{g^{-1}}(x))$ and $\Phi'(f)=\int_G \beta_g(\varphi'_0(f(g)))dg$.
Then $\Phi$ and $\Phi'$ are equivariant in the sense that $\Phi\circ \beta_g=(L_g\otimes \id)\circ \Phi$ and 
$\Phi'\circ (L_g\otimes \id)=\beta_g\circ \Phi'$ hold. 
For $x\in F$, we have 
\[\|\Phi'\circ \Phi(x)-x\|=\|\int_G\beta_g(\varphi'_0\circ \varphi_0(\beta_{g^{-1}}(x)))dg-x\|<\varepsilon/2.\]
Now the statement follows from Theorem \ref{Fejer}. 
\end{proof} 

\section{Proof of Theorem \ref{main2}}
We first establish an equivariant version of Kirchberg's dilation theorem \cite{Kir94}, \cite[Proposition 1.7]{KP}, 
as an application of Theorem \ref{main1} and Lemma \ref{eqCPAP}. 

\begin{theorem}\label{dilation} Under the assumption of Theorem \ref{main2}, we assume that $A$ is unital. 
Let $\nu:A\to A$ be a $G$-equivariant ucp map. 
Then for every finite set $F\subset A$ and $\varepsilon>0$, there exists $V$ in $A^\alpha$ such that 
$\|\nu(x)-V^*xV\|<\varepsilon$ for all $x\in F$.  
\end{theorem}

\begin{proof} Since $A$ is nuclear, Lemma \ref{eqCPAP} shows that there exist $n\in \N$, a unitary representation $u$ 
of $G$ in $\M_n(\C)$, and $G$-equivariant ucp maps $\varphi:A\to \M_n(\C)$ and $\varphi':\M_n(\C)\to A$ satisfying 
$\|\varphi'\circ \varphi(x)-\nu(x)\|<\varepsilon/2$ for all $x\in F$, where $\M_n(\C)$ is equipped with a $G$-action given by $\Ad u_g$. 

Thanks to Remark \ref{equivalence2}, there exist globally $\alpha$-invariant Hilbert spaces $\cH$ and $\overline{\cH}$ 
with orthonormal bases $\{\psi_i\}_{i=1}^n$ and $\{\bpsi_i\}_{i=1}^n$ respectively satisfying 
$\alpha_g(\psi_i)=\sum_{j}u_{ji}(g)\psi_j$, and $\alpha_g(\bpsi_i)=\sum_{j}\overline{u_{ji}(g)}\bpsi_j$. 
We let $F_1=\cup_{i,j}\bpsi_iF\bpsi_j^*$. 

Let $\{\xi_i\}_{i=1}^n$ be the standard basis of $\C^n$ and let $\{e_{ij}\}_{i,j}$ be the corresponding system of matrix units. 
We define a state $\omega$ of $A$ by 
\[\omega(x)=\frac{1}{n}\sum_{1\leq i,j\leq n}\inpr{\varphi(\bpsi_i^*x\bpsi_j)\xi_j}{\xi_i}.\]
Then a direct computation shows that $\omega$ is $\alpha$-invariant. 
We have 
\[\varphi'\circ \varphi(x)=n\sum_{i,j}\omega(\bpsi_i x \bpsi_j^*)\varphi'(e_{ij}).\]

We claim that there exists an isometry $v\in A^\alpha$ satisfying 
\[\|\omega(x)-v^*yv\|<\delta,\quad \forall y\in F_1,\]
where $\delta=\varepsilon/(2n^3)$ for our purpose. 
Thanks to Theorem \ref{main1}, the action $\alpha$ is quasi-product, and $\omega$ can be approximated 
by $\alpha$-invariant pure states in the weak$*$ topology (see \cite[Theorem 3.1]{BKR97}).  
Thus there exists an $\alpha$-invariant pure state 
$\omega_0$ of $A$ satisfying $|\omega_0(y)-\omega(y)|<\delta/2$ for all $y\in F_1$. 
Let $L=\{y\in A;\; \omega_0(y^*y)=0\}$. 
Then $L\cap L^*$ is a globally $\alpha$-invariant hereditary C$^*$-subalgebra of $A$, and 
there exists an approximate units $\{u_n\}_n$ of $L\cap L^*$ in $A^\alpha$ by Lemma \ref{appunit}. 
Let $a_n=1-u_n$. 
Then $\{a_n\}_n$ is a decreasing sequence in $A^\alpha_+$ of norm 1 such that  
$\{\|a_n(x-\omega_0(x))a_n\|\}_n$ converges to 0 for all $x \in A$ because 
$x-\omega_0(x)\in \ker \omega_0=L+L^*$. 
Since $A^\alpha$ is purely infinite, we can see from this that there exist isometries $\{v_n\}_n$ in $A^\alpha$ such that 
$\{v_n^*xv_n -\omega_0(x)\}_n$ converges to 0 for all $x\in A$.  
Thus the claim is shown. 

Next we claim that there exists an isometry $w\in A^\alpha$ satisfying 
\[\varphi'(e_{ij})=w^*\psi_i\psi_j^*w, \quad 1\leq i,j\leq n.\]
Since $\varphi'$ is completely positive, we have 
$a:=(\varphi'(e_{ij}))\in \M(A)_n$ is positive, and a direct computation shows that it is fixed by $\alpha_g\otimes \Ad \bu_g$. 
Since $b=a^{1/2}$ is fixed by $\alpha_g\otimes \Ad \bu_g$, we get
\[w:=\sum_{j,k}\psi_j\bpsi_k b_{kj}\in A^\alpha.\]
Now we have 
\[w^*\psi_i\psi_j^*w=\sum_kb_{ki}^*b_{kj}=\varphi(e_{ij}),\]
and the claim is shown. 

Let 
\[V=\sqrt{n}\sum_{j=1}^n \bpsi_j^*v\psi_j^*w,\]
which is an element in $A^\alpha$. 
Then for $x\in F$, 
\[\|\varphi'\circ\varphi(x)-V^*xV\|=\|n\sum_{i,j}w^*\psi_i(\omega(\bpsi_ix\bpsi_j^*)-v^*\bpsi_ix\bpsi_j^*v)\psi_j^*w\|<\varepsilon/2,\]
and the statement is shown. 
\end{proof}

\begin{proof}[Proof of Theorem \ref{main2}] Let $e\in A^\alpha$ be a non-zero projection with $[e]_0=0$ in $K_0(A)$. 
To show that $\alpha$ is isometrically shift-absorbing, it suffices to show that the restriction of $\alpha$ to $eAe$ has 
the same property. 
Thus we may and do assume $A^\alpha\in \cC_2$. 

We denote by $\alpha^\infty_g$ the automorphism of $A^\infty$ induce by $\alpha_g$. 
Since the infinite direct sum of the regular representation is equivalent to 
$\bigoplus_{\pi\in \hG} \pi^{\oplus \infty}$, 
it suffices to show that there exists a Hilbert space $\cK_\pi$ in $A^\infty\cap A'$ globally $\alpha^\infty$-invariant and 
equivalent to $\pi$ for each $\pi\in \hG$ (not necessarily with support 1), and there exists a unital embedding 
of $\cO_\infty$ into $(A^\alpha)^\infty \cap A'$. 
Indeed, we choose an orthonormal basis $\{\eta(\pi)\}_{i=1}^{d(\pi)}$ of $\cK_\pi$ for each $\pi\in \hG$, and isometries 
$\{S_{\pi,n}\}_{(\pi,n)\in \hG\times \N}$ in $(A^\alpha)^\infty\cap A'$ with mutually orthogonal ranges. 
Then 
\[\{S_{\pi,n}\eta(\pi)_i;\; \pi\in \hG,\; 1\leq i\leq d(\pi),\; n\in\N\}\]
give a desired set of generators of $\cO_\infty$ in $A^\infty \cap A'$. 

For each $\cH\in \cU(G,\alpha)$, we apply Theorem \ref{dilation} to $\rho_\cH$, and obtain an isometry 
$V_{\cH}\in (A^\alpha)^\infty$ satisfying $\rho_\cH(x)=V_{\cH}^*xV_{\cH}$ for all $x\in A$. 
We set $\cK_{\cH}=V_{\cH}\cH$, which is a Hilbert space in $A^\infty$, and the restriction of $\alpha^\infty$ to $\cK_{\cH}$ 
is equivalent to $\cH$. 
For $x\in A$, we have 
\[\psi(\cH)_i^*V_{\cH}^*xV_\cH\psi(\cH)_i=\psi(\cH)_i^*\rho_{\cH}(x)\psi(\cH)_i=x,\]
and 
\begin{align*}
\lefteqn{[x,V_{\cH}\psi(\cH)_i]^*[x,V_{\cH}\psi(\cH)_i]} \\
 &=\psi(\cH)_i^*V_{\cH}^*x^*xV_\cH\psi(\cH)_i-\psi(\cH)_i^*V_{\cH}^*x^*V_\cH\psi(\cH)_ix-x^*\psi(\cH)_i^*V_{\cH}^*xV_\cH\psi(\cH)_i
 +x^*x\\
 &=0.
\end{align*}
Thus the C$^*$-condition implies $\cK_{\cH}\subset A^\infty \cap A'$. 
We let $\cK_\pi=\cK_{\cH_\pi}$. 

When $\cH$ is equivalent to $1\oplus 1$, we have $\cK_{\cH}\subset (A^\alpha)^\infty\cap A'$. 
This shows that there exist two isometries with mutually orthogonal ranges in $(A^\alpha)^\infty\cap A'$, 
and there exists a unital embedding of $\cO_\infty$ in $(A^\alpha)^\infty\cap A'$. 
\end{proof}

\section{Quasi-free actions on the Cuntz algebras}
Before ending this paper, we discuss quasi-free actions of a compact group $G$ on the Cuntz algebras $\cO_n$ and $\cO_\infty$ 
as applications of our main results.  
We follow the convention in \cite{FLR00} for graph C$^*$-algebras. 

\begin{example} Let $\rho$ be a faithful unitary representation of a compact group $G$ in $\M_n(\C)$, and let $\alpha$ be 
the quasi-free action of $G$ arising from $\rho$ on $\cO_n$. 
It is known that $\cO_n\cap ({\cO_n}^{U(n)})'$ is trivial (see \cite[Theorem 3.2]{BE86}, \cite[Corollary 3.3]{DR87}), 
and hence $\alpha$ is minimal. 
Thus if ${\cO_n}^\alpha$ is simple, Theorem \ref{main1} implies that $\alpha$ is quasi-product. 
In fact, \cite[Theorem 3.1]{DR87} shows that ${\cO_n}^\alpha$ is simple if $\rho(G)\subset SU(n)$. 
We can relax this assumption to the following: 
\begin{itemize}
\item[($*$)] For every $\pi\in \hG$, there exists $n\geq 0$ such that $\pi$ is contained in $\rho^{\otimes n}$. 
\end{itemize}
\cite[Theorem 7.1]{KPRR97} shows that ${\cO_n}^\alpha$ is isomorphic to a corner of the graph algebra $C^*(\cG_\rho)$ of 
the following graph $\cG_\rho$: 
the vertex set $\cG_\rho^0$ is $\hG$, and the number of edges from $\sigma\in \hG$ to $\pi\in \hG$ is 
$\dim \Hom_G(\pi,\rho\otimes \sigma)$. 
Now \cite[Theorem 3, 4]{FLR00} imply that ${\cO_n}^\alpha$ is purely infinite 
and simple under the condition ($*$), and our main results show that $\alpha$ is isometrically shift-absorbing. 
It is very likely that ($*$) is also necessary for ${\cO_n}^\alpha$ to be purely infinite and simple, but it does not seem 
to be known yet. 
\end{example}

For the Cuntz algebra $\cO_\infty$, we have the following statement, previously known only for abelian $G$ 
(see \cite[Proposition 7.4]{Ka01}, \cite[Theorem 5.1]{K80}, \cite[Corollary 4.17]{Mu24},  
and also \cite{CP92} for a related result). 

\begin{proposition} Let $(\rho,H)$ be a faithful unitary representation of a compact group $G$ on a separable infinite dimensional 
Hilbert space $H$, and let $\alpha$ be the corresponding quasi-free action of $G$ on $\cO_\infty$.  
Then the following conditions are equivalent:
\begin{itemize}
\item[$(1)$] $\alpha$ is isometrically shift-absorbing.
\item[$(2)$] ${\cO_\infty}^\alpha$ is purely infinite and simple.
\item[$(3)$] ${\cO_\infty}^\alpha$ is simple.
\item[$(4)$] The Condition $(*)$ holds. 
\end{itemize}
\end{proposition}

\begin{proof} Let 
\[(\cF(H),\cF(\rho))=(\bigoplus_{n=0}^\infty H^{\otimes n},\bigoplus_{n=0}^\infty \rho^{\otimes n}).\] 
By convention $H^{\otimes 0}=\C\Omega$ where $\Omega$ is the vacuum vector. 
We regard $\cO_\infty$ as a concrete C$^*$-algebra acting on the full Fock space $\cF(H)$ generated by the left creation operators 
$l(\xi)$, $\xi\in H$. 
Then the quasi-free action $\alpha$ is the restriction of $\Ad \cF(\rho)$ to $\cO_\infty$. 
Since the vacuum state is $\alpha$-invariant and pure on $\cO_\infty$, if ${\cO_\infty}^\alpha$ simple,  
the action $\alpha$ is quasi-product, and in particular it is is minimal.  

(1) $\Rightarrow$ (2). This is shown in \cite[Example 4.5]{Mu24}. 

(2) $\Rightarrow$ (1). This follows from Theorem \ref{main2}. 

(2) $\Rightarrow$ (3). This is trivial.   

(3) $\Rightarrow$ (4). Since ${\cO_\infty}^\alpha$ is simple and $\alpha$ is minimal, 
the crossed product $\cO_\infty\rtimes_\alpha G$ is simple (see \cite[Proposition A]{Mu24}). 
Thus the defining representation of $\cO_\infty$ and $\cF(\rho)$ give a faithful representation of 
$\cO_\infty\rtimes_\alpha G$ on $\cF(H)$, and hence $\cF(\rho)$ gives rise to a faithful representation of the group 
C$^*$-algebra of $G$, which implies (4). 

(4) $\Rightarrow$ (2). We show the implication by identifying ${\cO_\infty}^\alpha$ with a corner of $C^*(\cG_\rho)$ as 
in the case of $\cO_n$. 
Since the argument in \cite{KPRR97} does not work in this case, we give a more direct argument here. 
We let $K=\bigoplus_{\pi\in \hG}H_\pi$, and equip $\K(K)$ with the $G$-action given by $\Ad(\bigoplus_{\pi\in \hG}\pi)$. 
We first embed $C^*(\cG_\rho)$ into $(\cO_\infty\otimes \K(K))^G$. 
 
For each $\pi,\sigma\in \hG$, we fix an orthonormal basis $\{e\}$ of  $\Hom_G(H_\pi,H\otimes H_\sigma)$, where the inner product 
of $e$ and $f$ is given by $\inpr{e}{f}=f^*e\in \End_G(H_\pi)=\C 1_\pi$. 
We identify $\{e\}$ with the set of edges from $\sigma$ to $\pi$, and $\cup_{\pi,\sigma\in \hG}\{e\}$ 
with the edge set $\cG_\rho^1$. 
For $e$, we define $S_e$ to be the corresponding element in 
\[\left(l(H)\otimes (H_\sigma\otimes H_\pi^*)\right)^G\subset (\cO_\infty\otimes \K(K))^G,\]
where $H_\sigma\otimes H_\pi^*$ is identified with a subspace of $\K(K)$ in a natural way.  
Then we have $S_e^*S_e=1\otimes p_\pi$ and $S_eS_e^*\leq 1\otimes p_\sigma$, 
where $p_\pi$ is the projection from $K$ onto $H_\pi$.  
If $e_1,e_2,\ldots,e_n\in \cG_\rho^1$ are distinct edges with a common source $\sigma$, we have 
$S_{e_i}^*S_{e_j}=\delta_{i,j}1\otimes p_{r(e_i)}$, which shows that $\{S_{e_i}S_{e_i}^*\}_{i=1}^n$ are mutually orthogonal 
subprojections of $1\otimes p_\sigma$. 
Thus $\{S_e,1\otimes p_\pi\}$ form a Cuntz-Krieger $\cG_\rho$-family, and there exists a surjection 
from the graph algebra $C^*(\cG_\rho)$ onto the C$^*$-algebra $B$ generated by $\{S_e\}_{e\in \cG_\rho^1}$. 
\cite[Theorem 3, 4]{FLR00} show that $C^*(\cG_\rho)$ is purely infinite and simple thanks to the condition 
($*$), and so is $B$.  
Now our task is to show $(1\otimes p_1)B(1\otimes p_1)={\cO_\infty}^\alpha\otimes p_1$, which implies (2). 

Let $\cG_\rho(1,\pi;n)$ be the set of paths $\xi$ in $\cG_\rho$ of length $n$ satisfying $s(\xi)=1$ and $r(\xi)=\pi$. 
Then the linear span of $\{S_\xi\}_{\xi\in \cG_\rho(1,\pi;n)}$ is dense in $(l(H)^n\otimes (H_1\otimes H_\pi^*))^G$ by construction,
and so the linear span of 
\[\{S_\xi S_\eta^*;\;\xi\in \cG_\rho(1,\pi;n),\;\eta\in \cG_\rho(1,\pi;m),\; \pi\in \hG\}\] 
is dense in $(l(H)^n {l(H)^m}^*)^G\otimes p_1$. 
Thus $(1\otimes p_1)B(1\otimes p_1)={\cO_\infty}^\alpha\otimes p_1$. 
\end{proof}

The following is a consequence of our main results and \cite[Corollary 6.4]{GS22-2}, 
and it is a generalization of \cite[Theorem 5.1]{GI11} in the finite group case. 

\begin{cor} For a compact group $G$, any two faithful quasi-free actions of $G$ on the Cuntz algebras $\cO_\infty$ satisfying 
the condition $(*)$ are mutually conjugate.  
Such actions are absorbed through tensor product by every faithful minimal action of $G$ on a Kirchberg algebra whose fixed point 
algebra is a Kirchberg algebra. 
(Note that when $G$ is finite or $SU(n)$, the condition $(*)$ is automatically satisfied.)
\end{cor}

\section*{Acknowledgments}
The author would like to thank Miho Mukohara for stimulating discussions and Narutaka Ozawa 
for useful discussions on the CPAP. 


\end{document}